\documentclass[10pt]{amsart}

\usepackage{amsthm}
\usepackage{amssymb}
\usepackage{amsmath}
\usepackage{mathtools}
\usepackage{pdfpages}
\usepackage{float}
\usepackage[colorlinks=true,pdftex,unicode=true,linktocpage,bookmarksopen,hypertexnames=false]{hyperref}
\usepackage{tikz-cd}
\usepackage{caption}
\usepackage[capitalise]{cleveref}
\usepackage{quiver}

\DeclareMathOperator{\id}{id}

\DeclareMathOperator{\Ker}{Ker}

\DeclareMathOperator{\Soc}{Soc}

\DeclareMathOperator{\op}{op}
\DeclareMathOperator{\ab}{ab}
\DeclareMathOperator{\Res}{Res}
\DeclareMathOperator{\Inf}{Inf}
\DeclareMathOperator{\Tra}{Tra}
\DeclareMathOperator{\Img}{Im}

\newcommand{\cD}{\mathcal{D}}

\newcommand{\Z}{\mathbb{Z}}
\newcommand{\C}{\mathbb{C}}

\newcommand{\B}{\mathcal{B}}

\newcommand{\Aut}{\operatorname{Aut}}
\newcommand{\Hom}{\operatorname{Hom}}

\newcommand{\End}{\mathrm{End}}
\newcommand{\PGL}{\mathrm{PGL}}

\newcommand{\GL}{\mathrm{GL}}

\newcommand{\Ann}{\operatorname{Ann}}

\newcommand{\Ext}{\operatorname{Ext}}

\makeatletter
\numberwithin{equation}{section}
\numberwithin{figure}{section}
\numberwithin{table}{section}

\newtheorem{thm}{Theorem}[section]
\newtheorem*{thm*}{Theorem}
\newtheorem{lem}[thm]{Lemma}
\newtheorem{cor}[thm]{Corollary}
\newtheorem{pro}[thm]{Proposition}
\newtheorem{notation}[thm]{Notation}

\newtheorem{defn}[thm]{Definition}

\newtheorem{convention}[thm]{Convention}
\newtheorem{rem}[thm]{Remark}
\newtheorem{exa}[thm]{Example}
\makeatother

\title{Schur covers of skew braces}
\author{T. Letourmy}
\author{L. Vendramin}

\address{Department of Mathematics and Data Science, Vrije Universiteit Brussel, Pleinlaan 2, 1050 Brussel, Belgium}
\email{Leandro.Vendramin@vub.be}

\address{Départament de Mathéematique, Université Libre de Bruxelles, Boulevard du Triomphe, B-1050 Brussels, Belgium; and 
Department of Mathematics and Data Science, Vrije Universiteit Brussel, Pleinlaan 2, 1050 Brussel, Belgium}
\email{thomas.letourmy@ulb.be}


\begin{document}

\begin{abstract}
    We develop the theory of Schur covers of finite skew braces. We prove
    the existence of at least one Schur cover. We also 
    compute several examples. We 
    prove that different Schur covers are
    isoclinic. Finally, we prove that Schur covers
    have the lifting property concerning
    projective representations of skew braces. 
\end{abstract}

\maketitle

\setcounter{tocdepth}{1}

\section*{Introduction}

In the combinatorial theory 
of the Yang--Baxter equation, several algebraic structures
play a fundamental role: groups of I-type \cite{MR1637256}, 
bijective 1-cocycles \cite{MR1722951,MR1809284}, 
matched pair of groups \cite{MR2383056} 
braided groups \cite{MR1769723} and, 
remarkably, skew braces \cite{MR3647970,MR2278047}. Skew 
braces have their germ in the theory of
Jacobson radical rings 
and show connections with several
areas of mathematics such as
triply factorized groups and Hopf--Galois structures; see
for example \cite{MR4427114,MR3291816,MR3763907}. 

A \emph{skew brace} is a triple $(A,+,\circ)$, where $(A,+)$ 
and $(A,\circ)$ are (not necessarily abelian) groups and
\[
a\circ (b+c)=a\circ b-a+a\circ c
\]
holds for all $a,b,c\in A$.

Skew braces produce solutions to the set-theoretic 
Yang--Baxter
equation. Furthermore, all 
set-theoretic solutions come from skew braces; see 
for example \cite{MR3763907}. 

In some sense, skew braces are ring-like objects. There exists the notion of 
ideals (and hence quotients) of
a skew brace. For example, if $A$ is a skew brace, the socle
$\Soc(A) = \{a\in A : -a+a\circ b = b\;\forall b\in A\}\cap Z(A,+)$ 
and the annihilator $\Ann(A)=\Soc(A)\cap Z(A,\circ)$ 
of $A$ are both ideals of $A$. 

Although there is a rich structure theory of skew braces, 
several aspects are still under development. A prominent example 
is the cohomology theory defined in 
\cite{MR3530867,MR3558231}. 
In the literature, one finds papers where 
different theories of extensions 
of skew braces are considered; see for example 
\cite{MR3763276,MR3917122}. 
These ideas were later 
generalized \cite{MR4604853} to a theory of 
extensions of skew braces by abelian groups. 

The annihilator of 
a skew brace (which is somewhat the center of a skew brace) 
plays an important role in the theory; 
see for example \cite{MR4663905,MR4627847}. 
Our goal is to develop a not-yet-explored  
application of the cohomology of skew braces. We deal, in particular, 
with some universal annihilator extensions. These ideas 
are deeply connected to the representation theory 
of skew braces. 

Schur introduced Schur covers in \cite{MR1580818} 
to classify projective representations of groups. A Schur cover of a group is a certain
universal central extension. This object's strength is based on
the so-called \emph{lifting property}. The calculation
of Schur covers is still a significant problem; for example, the book 
\cite{MR1205350} considers the case of the symmetric group. 

The role that skew braces play in several mathematical theories
suggests that a deep understanding of cohomology and representation 
theory of skew braces will make an impact in areas such as non-commutative ring theory \cite{MR4023387,MR3814340}, 
the combinatorial study of the Yang--Baxter equation \cite{MR2584610,MR3177933,MR3861714}, 
and Hopf--Galois structures \cite{MR4390798,MR3763907}. 
For that reason, studying
representations, cohomology, and, in particular, 
Schur covers of skew braces 
seems to be of fundamental importance. 

This work proves that every finite skew brace admits at least one Schur cover (Theorem \ref{thm:existence}). 
For this purpose, we develop the theory of Schur multipliers. This includes
basic results (e.g. Theorems \ref{thm:M(Q)exp} and \ref{thm:M(Q)finite}) and
a Hoschild--Serre-type exact sequence (Theorem \ref{thm:HoschildSerre}). 
Several concrete examples of Schur multipliers and covers 
are computed (e.g. Theorems \ref{thm:M(Gop)}, \ref{thm:Schur_cyclic_abelian} 
and \ref{thm:directproductbicyclic} and Propositions  \ref{pro:perfectgroup} and \ref{pro:coverbicyclic}). 
In group theory, Schur covers 
are not unique up to isomorphism; they are isoclinic \cite{MR0003389}. Isoclinism of skew braces 
was defined in \cite{zbMATH07779852}. We prove here that
any two Schur covers of a finite skew brace 
are isoclinic (Theorem \ref{thm:isoclinism}). Finally, 
inspired by \cite{MR4504147}, we introduce projective 
representations of skew braces 
and prove that Schur covers have the lifting property 
(Theorem \ref{thm:lift}).

\section{Preliminaries}
\label{preliminaries}

A \emph{skew brace} is a triple $(A,+,\circ)$, where $(A,+)$ 
and $(A,\circ)$ are (not necessarily abelian) groups and
\[
a\circ (b+c)=a\circ b-a+a\circ c
\]
holds for all $a,b,c\in A$.

If $A$ is a skew brace, 
the group $(A,+)$ 
will be called the \emph{additive} group
of $A$ and denoted by $A_+$ and 
$(A,\circ)$ will be called the 
\emph{multiplicative} group of $A$ and denoted by $A_\circ$. 
If $x\in A$, then 
$x'$ will denote the inverse of $x$ with respect to the circle operation. 

A skew brace is said to be of \emph{abelian type} if its additive group is abelian. 

Let $A$ and $B$ be skew braces. 
A skew brace \emph{homomorphism} is a map $f\colon A\to B$ such that
$f(x+y)=f(x)+f(y)$ and $f(x\circ y)=f(x)\circ f(y)$ for all $x,y\in A$. 

\begin{exa}
\label{exa:trivial}
    Any group $G$ is a trivial skew brace
    with the operations $x+y=xy$ and $x\circ y=xy$. 
\end{exa}

\begin{convention}
    Let $G$ be a group. When treating $G$ as a skew brace, we will
    refer to the trivial skew brace structure of Example \ref{exa:trivial}. 
    Conversely, if $(B,+,\circ)$ is a trivial
    skew brace, we will often treat $B$ as the group $(B,+)$. 
\end{convention}

\begin{exa}
   Let $G$ be a group. The operations $x\circ y=yx$ and 
    $x+y=xy$ turn $G$ into a skew brace; it 
    is known as the \emph{almost trivial skew brace} over $G$.
\end{exa}

\begin{exa}
    Let $R$ be a radical ring. This means that 
    $(R,\circ)$, where
    \[
    x\circ y=x+xy+y,
    \]
    is a group. 
    Then $(R,+,\cdot)$ is a skew brace
    of abelian type such that
    \[
    (x+y)\circ z=x\circ z-z+y\circ z
    \]
    holds for all $x,y,z\in R$. 
\end{exa}

\begin{exa}
\label{exa:A(n,d)}
    Let $n\geq2$ be an integer and $d$ a divisor of $n$ such that every prime divisor of $n$ divides $d$. 
    Then $\Z/(n)$ with the usual addition modulo $n$ and 
    \[
    x\circ y=x+y+dxy
    \]
    is a skew brace of abelian type. 
    This brace will be denoted
    by $C_{(n,d)}$. Note that $d=|\Soc(C_{(n,d)})|$.
\end{exa}

If $A$ is a skew brace, then the map 
    $\lambda\colon (A,\circ)\to\Aut(A,+)$,
    $a\mapsto\lambda_a$,
    where $\lambda_a(b)=-a+a\circ b$, 
    is a group homomorphism. 
An \emph{ideal} $I$ of $A$ is 
a subset of $A$ such that $(I,+)$ is a normal subgroup
of $(A,+)$, $(I,\circ)$ is a normal subgroup of $(A,\circ)$ 
and $\lambda_a(I)\subseteq I$ for all $a\in A$.
We write $A^{(2)}=A*A$ to denote the additive subgroup of $A$ generated 
by $\{a*b:a,b\in A\}$, where
$a*b=-a+a\circ b-b$. 
One proves that $A^{(2)}$ is an ideal of $A$. If $a,b\in A$, 
we write $[a,b]_+=a+b-a-b$. We write $[A,A]_+$ to
denote the commutator subgroup of the additive group of $A$.  
The \emph{commutator} $A'$ of $A$ is
    the additive subgroup of $A$ generated by $[A,A]_+$ and $A^{(2)}$. 
    It can be proved that $A'$ is an ideal of $A$ and that it is the smallest one such that $A^{\ab}=A/A'$ is an abelian group. 

\subsection{Group cohomology} 

Let $G$ be a group and $K$ an abelian group. The group $Z^2(G,K)$ of \emph{factor sets} (or 2-cocycles) consists of maps $f\colon G\times G\to K$ 
such that $f(1,y)=f(x,1)=0$ and 
\[ 
f(y,z)-f(xy,z)+f(xy,z)-f(x,y)=0
\]
for all $x,y,z\in G$. Let 
$B^2(G, K)$ be the group of \emph{2-coboundaries}, which is the subgroup of maps 
$f\in Z^2(G,K)$ that can be written as 
\[
f(x,y)=g(y)-g(xy)+g(x)
\]
for some map $G\to K$. We say that two factor sets are \emph{cohomologous} if their difference is a $2$-coboundary.
The \emph{second cohomology group} of $G$ with coefficients in $K$ is defined 
as the group 
\[
H^2(G,K)=Z^2(G,K)/B^2(G,K).
\]
For $f\in Z^2(G,K)$, we write 
$\overline{f}=f+B^2(G,K)$. 

A \emph{central extension} 
of a group $G$ by an abelian group $K$ is a group $E$ together
with a short exact sequence of groups 
     \begin{equation}
         \begin{tikzcd}
	        0 & K & {E} & G & 0
	        \arrow[from=1-1, to=1-2]
	        \arrow["\iota", from=1-2, to=1-3]
	        \arrow["\pi", from=1-3, to=1-4]
	        \arrow[from=1-4, to=1-5]
        \end{tikzcd}
     \end{equation}
    such that $\iota(K)\subseteq Z(E)$. 

    We say that two central extensions $E$ and $E_1$ of $G$ by $K$ 
    are \emph{equivalent} if 
    there exists a group homomorphism $\gamma\colon E\to E_1$ such that
    the diagram
    \begin{equation}
    \begin{tikzcd}
    && E \\
	0 & K && G & 0 \\
	&& {E_1}
	\arrow[from=2-1, to=2-2]
	\arrow[from=2-4, to=2-5]
	\arrow[from=2-2, to=3-3]
	\arrow[from=3-3, to=2-4]
	\arrow["\gamma", from=1-3, to=3-3]
	\arrow[from=2-2, to=1-3]
	\arrow[from=1-3, to=2-4]
    \end{tikzcd}
    \end{equation}
    is commutative. 
    
    Let $\Ext(G,K)$ be the 
    set of equivalence classes of central extensions of $G$ by $K$. There
    is a bijective correspondence between 
    $H^2(G,K)$ and $\Ext(G,K)$. See for example \cite[11.1.4]{MR1357169} for details.

Let $N$ and $T$ be finite groups. 
Recall that a factor set 
$f\in Z^2(N\times T,\C^{\times})$ is said to be \emph{normal} 
if 
$f(n,t)=0$ for all $n\in N$ and $t\in T$, 
where $N$ is identified with $N\times\{1\}$ and 
$T$ with $\{1\}\times T$.

\begin{lem}
\label{lem:K}
    Let $N$ and $T$ be finite groups. 
    
    \begin{enumerate}
        \item Each $f\in Z^2(N\times T,\C^{\times})$ is cohomologous to a normal factor set $g$ such that $g|_{T\times T}=f|_{T\times T}$. 
        \item Every normal factor set $g$ is such that
            \[ 
            g(nt,n't')= g(n,n')g(t,n')g(t,t')
            \]
            for all $n,n'\in N$ and $t,t'\in T$.
        \item Given $\lambda\colon N\times N\to \C^{\times}$, $\sigma\colon T\times N\to \C^{\times}$ and $\rho\colon 
        T\times T\to \C^{\times}$, the map
        \[
        f(nt,n't')= \lambda(n,n')\sigma(t,n')\rho(t,t')
        \]
        is a normal factor set if and only if $\lambda\in Z^2(N,\C^{\times})$, $\rho\in Z^2(T,\C^{\times})$ and $\sigma \in \Hom(T,\Hom(N,\C^{\times}))$.
    \end{enumerate}
\end{lem}

\begin{proof}
    See \cite[Lemmas 2.2.3 and 2.2.4]{MR1200015}.
\end{proof}

For a finite group $G$, let $M(G)=H^2(G,\C^\times)$ denote the Schur multiplier of $G$.

The following theorem is known as the Schur--K\"uneth formula. We will use the notation $G^{\ab}=G/[G,G]$ for the \emph{abelianization} of a group
$G$.

\begin{thm}
Let $N$ and $T$ be finite groups. Then the map
\begin{align*}
\phi\colon M(N)\times T^{\ab}\otimes N^{\ab}\times M(T) 
\to M(N\times T)
\end{align*}
given by $\phi\left( \overline{f_{N}},f_{T\times N},\overline{f_T}\right)= \overline{f}$, where 
\[
f(nt,ms)= f_{N}(n,m)f_{N\times T}(t,m)f_{T}(t,s)
\]
is an isomorphism and $T^{\ab}\otimes N^{\ab}$ is identified 
with $\Hom(T,\Hom(N,\C^{\times}))$.
\label{thm:directproduct}
\end{thm}




\begin{proof}
    Lemma \ref{lem:K} yields a map 
    \[
    \psi\colon Z^2(N,\C^{\times})\times \Hom(T,\Hom(N,\C^{\times}))\times Z^2(T,\C^{\times})\to Z^2(G,\C^{\times}).
    \]
    We claim that the map $f=\psi(\lambda,\sigma,\rho)$ is a coboundary if and only if $\lambda$ and $\rho$ are coboundaries and $\sigma(t)(n)=1$ for all $t\in T$ and $n\in N$. 
    Assume that $f$ is a coboundary. There exists a map $h\colon G\to \C^{\times}$ such that $f=\partial h$. Thus
    \begin{align*}
        &\lambda(n,n')=h(n)h(n+n')^{-1}h(n'),
        &&\rho(t,t')=h(t)h(t+t')^{-1}h(t')
    \end{align*}
    for all $t,t'\in T$ and $n,n'\in N$. In addition, 
    \[
    \sigma(t,n)=f(t,n)= h(n)h(nt)^{-1}h(t)=f(n,t)=1.
    \]
    
    Conversely, assume that $\sigma$ is trivial and there exist maps $l\colon N\times N\to \C^{\times}$ and $k\colon T\times T\to \C^{\times}$ such that $\lambda=\partial l$ and $\rho=\partial k$. Let
    $h\colon G\to \C^{\times}$, $h(n,t)=l(n)k(t)$. Then $f=\partial h$.

    Now the map $\psi$ factors through $\phi$ and hence $\phi$ is injective. 
    Furthermore, $\phi$ is surjective by Lemma \ref{lem:K}. 
\end{proof}

\subsection{Annihilator extensions}

We recall (without proofs) a particular case of the theory of 
extensions of skew braces introduced in \cite{MR4604853}. 

In this section, $K$ will always denote an abelian group written multiplicatively. 

\begin{defn}
    Let $f\colon A\to Q$ be a skew brace homomorphism. We define a
    \emph{transversal} 
    of $f$ as a map $t\colon Q\to A$ such that $ft=\id_Q$ and 
    $t(0)=1$. 
\end{defn}

\begin{defn}
    A \emph{datum} $\cD$ is a pair $(Q,K)$, where 
    $Q$ is a skew brace, $K$ is an abelian group.
\end{defn}

\begin{notation}
    If $\cD=(Q,K)$ is a datum, 
    $\cD_+=(Q_+,K)$ and $\cD_\circ=(Q_\circ,K)$. 
\end{notation}

\begin{defn}
    An \emph{annhilator extension} 
    realizing the datum $\cD = (Q,K)$ is a skew brace $E$ together
    with a short exact sequence 
     \begin{equation}
         \label{eq:SE}
         \begin{tikzcd}
	        0 & K & {E} & Q & 0
	        \arrow[from=1-1, to=1-2]
	        \arrow["\iota", from=1-2, to=1-3]
	        \arrow["\pi", from=1-3, to=1-4]
	        \arrow[from=1-4, to=1-5]
        \end{tikzcd}
     \end{equation}
    such that $\iota(K)\subseteq\Ann(E)$. In particular, 
    \begin{equation}
    \label{eq:extaction}\iota(k)=t(x)+\iota(k)-t(x) \quad\text{and}\quad
    \iota(k)=t(x)\circ \iota(k)- t(x)
    \end{equation}
    for all $k\in K$ and 
    $x\in Q$ and all transversal 
    $t\colon Q\to E$ of $\pi$. 
\end{defn}
\begin{rem}
\label{rem: 1transversalenough}
    If $t$ and $t'$ are both transversals of $\pi$ in \eqref{eq:SE}, then
    \begin{align*}
        t(x)+\iota(k)-t(x)&=t'(x)+\iota(k)-t'(x),\\
        t(x)\circ \iota(k)- t(x) &= t'(x)\circ \iota(k)- t'(x)
    \end{align*}
    for all $x\in Q$. 
    That is \eqref{eq:extaction} may be reformulated as follows: there exists a transversal $t\colon Q\to E$ of $\pi$ such that 
    \[
    \iota(k)=t(x)+\iota(k)-t(x)\quad
    \text{and}\quad
    \iota(k)=t(x)\circ \iota(k)- t(x)
    \]
    for all $k\in K$ and 
    $x\in Q$. 
\end{rem}
\begin{rem}

\label{rem:decomposition}
    If an annihilator extension \eqref{eq:SE} realizes the 
    datum $(Q,K)$, 
    then every $a\in E$ can be written 
    uniquely as 
    $a=\iota(k)+t(x)$ 
    for some $k\in K$ and $x\in Q$. Moreover, 
    $a=\iota(k)+t(x)=\iota(k)\circ t(x)$, 
    as $\iota(K)\subseteq\Ann(E)$. 
\end{rem}

\begin{defn}
\label{defn:FS}
    A \emph{factor set (or 2-cocycle)}
    of the datum $(Q,K)$ is a pair $(\alpha,\mu)$, where 
    $\alpha,\mu\colon Q\times Q\to K$ are such that
    \begin{align}
        \label{eq:FS1}&\alpha(0,y)=\alpha(x,0)=1,\\
        \label{eq:FS2}&\mu(0,y)=\mu(x,0)=1,\\
        \label{eq:FS3}&\alpha(y,z)\alpha(x+y,z)^{-1}\alpha(x,y+z)\alpha(x,y)^{-1}=1,\\
        \label{eq:FS4}&\mu(y,z)\mu(x\circ y,z)^{-1}\mu(x,y\circ z)\mu(x,y)^{-1}=1
    \end{align}
    hold for all $x,y,z\in Q$ 
    and the compatibility condition
    \begin{equation}
    \begin{aligned}
        \label{eq:FS5}\alpha(y,z)&\mu(x,y+z)
        \alpha(-x,x)=
        \mu(x,z)
        \alpha(x\circ y,-x)\alpha(x\circ y-x,x\circ z)\mu(x,y)
    \end{aligned}
    \end{equation}
    holds for all $x,y,z\in Q$. 
\end{defn}
\begin{rem}
    By \eqref{eq:FS3},
    \begin{align*}
    &\alpha(x\circ y,-x)^{-1}\alpha(x\circ y-x,x\circ z)^{-1}= \alpha(-x,x\circ z)^{-1}\alpha(x\circ y,\lambda_x(z))^{-1}\\
    &\alpha(x,-x)\alpha(-x,x\circ z)^{-1}= \alpha(x,\lambda_x(z)).
    \end{align*}
    Thus \eqref{eq:FS5} is equivalent to 
    \begin{equation}
    \label{eq:equivFS5}
    \begin{aligned}
        \alpha(y,z)\alpha(x\circ y,\lambda_x(z))^{-1}\alpha(x,\lambda_x(z))
        =\mu(x,y)\mu(x,y+z)^{-1}\mu(x,z).
    \end{aligned}
    \end{equation}
\end{rem}


\begin{notation}
    We write $Z^2_b(\cD)$ to denote 
    the set of all factors sets of the datum $\cD=(Q,K)$. 
\end{notation}

The set $Z^2_b(\cD)$ is an abelian group with
\[
(\alpha,\mu)(\alpha_1,\mu_1)=(\alpha\alpha_1,\mu\mu_1),
\]    
where $\alpha\alpha_1$ and $\mu\mu_1$ are defined point-wise. 

\begin{pro}
   If an extension \eqref{eq:SE} realizes the datum 
   $\cD=(Q,K)$ and $t\colon Q\to E$ is a transversal of $\pi$, then
   \begin{align*}
       &\alpha\colon Q\times Q\to K, && (x,y)\mapsto t(x)+t(y)-t(x+y),\\
       &\mu\colon Q\times Q\to K, && (x,y)\mapsto t(x)\circ t(y)\circ t(x+y)', 
   \end{align*}
   satisfy conditions 
   \eqref{eq:FS1}--\eqref{eq:FS5} of Definition \ref{defn:FS}. 
\end{pro}


\begin{pro}
    Let $(\alpha,\mu)$ be a factor set of $\cD=(Q,K)$. 
    Then $K\times Q$ with operations 
    \[
         (a,x)+(b,y)=(ab\alpha(x,y),x+y),\quad 
         (a,x)\circ (b,y)=(ab\mu(x,y),x\circ y), 
     \]
     is an extension realizing the datum $\cD$. This skew brace
     structure on $K\times Q$ will be denoted by 
     $K\times_{(\alpha,\mu)} Q$. 
\end{pro}

    

\begin{notation}
    Let $(Q,K)$ be a datum.
    Given a map $h\colon Q\rightarrow K$, we define 
    \begin{align*}
        &\partial_+ h\colon Q\times Q\to K,\quad (x,y) \mapsto h(y)h(x+y)^{-1}h(x),\\
        &\partial_{\circ} h\colon Q\times Q\to K,\quad (x,y) \mapsto h(y)h(x\circ y)^{-1}h(x).
    \end{align*}
\end{notation}

\begin{defn}
    A \emph{coboundary} of the datum $(Q,K)$ is a
    pair $\alpha\colon Q\times Q\to K$ and 
    $\mu\colon Q\times Q\to K$ such that there 
    is a map $h\colon Q\to K$ with $h(0)=1$ such that 
    \begin{align*}
        &\alpha(x,y)=\partial_+h(x,y)\quad\text{and}\quad
        \mu(x,y)=\partial_{\circ}h(x,y)        
    \end{align*}
    hold for all $x,y\in Q$. 
\end{defn}

\begin{notation}
    We write $B^2_b(\cD)$ to denote the
    set of coboundaries of the datum $\cD=(Q,K)$. 
\end{notation}

\begin{rem}
    A coboundary of the datum $\cD$ is a factor set. 
\end{rem}

\begin{defn}
\label{defn:H2}
    The 2nd \emph{cohomology} group of $\cD=(Q,K)$ 
    is defined as the quotient 
    $H^2_b(\cD)=Z^2_b(\cD)/B^2_b(\cD)$.
\end{defn}

\begin{defn}
    Let $E$ and $E_1$ be two annihilator extensions realizing 
    $(Q,K)$. 
    We say that $E$ and $E_1$ are \emph{equivalent} if 
    there exists a skew brace homomorphism $\gamma\colon E\to E_1$ such that
    the diagram
    \begin{equation}
    \label{eq:equivalence}
    \begin{tikzcd}
    && E \\
	0 & K && Q & 0 \\
	&& {E_1}
	\arrow[from=2-1, to=2-2]
	\arrow[from=2-4, to=2-5]
	\arrow[from=2-2, to=3-3]
	\arrow[from=3-3, to=2-4]
	\arrow["\gamma", from=1-3, to=3-3]
	\arrow[from=2-2, to=1-3]
	\arrow[from=1-3, to=2-4]
    \end{tikzcd}
    \end{equation}
    is commutative. 
\end{defn}

\begin{rem}
    If the annihilator extensions $E$ and $E_1$ are equivalent, 
    the map $\gamma$ is indeed a skew brace isomorphism.  
\end{rem}

\begin{pro}
    Let $E$ and $E_1$ be annihilator extensions realizing the datum $(Q,K)$ 
    with respective factor sets $(\alpha,\mu)$ and $(\alpha_1,\mu_1)$. Then
    $E$ and $E_1$ are equivalent if and only if $(\alpha\alpha_1^{-1},\mu\mu_1^{-1})$ 
    is a coboundary.
\end{pro}


We summarize the results of this section in the following statement:

\begin{thm}
\label{thm:bijection}
    There is a bijective correspondence between 
    equivalence classes of annihilator extensions 
    \[
    \begin{tikzcd}
	        0 & K & {E} & Q & 0
	        \arrow[from=1-1, to=1-2]
	        \arrow[from=1-2, to=1-3]
	        \arrow[from=1-3, to=1-4]
	        \arrow[from=1-4, to=1-5]
        \end{tikzcd}
    \]
    realizing the datum $\cD=(Q,K)$ 
    and elements of $H^2_b(\cD)$. 
\end{thm}

As a consequence of Theorem \ref{thm:bijection}, it follows that
every annihilator extension realizing the datum $\cD=(Q,K)$ 
is equivalent to 
    \[
    \begin{tikzcd}
	        0 & K & K\times_{(\alpha,\mu)} Q & Q & 0
	        \arrow[from=1-1, to=1-2]
	        \arrow[from=1-2, to=1-3]
	        \arrow[from=1-3, to=1-4]
	        \arrow[from=1-4, to=1-5]
        \end{tikzcd}
    \]
for some $(\alpha,\mu)\in Z_b^2(\cD)$. 

\subsection{Cohomology of trivial and almost trivial skew braces}

\begin{defn}
    Let $\cD=(Q,K)$ be a datum. 
    The set of 
    \emph{pre-coboundaries} is defined as the subset 
    $P^2(\cD)$ of $Z^2_{b}(\cD)$ 
    that consists of all the pairs $(\partial_+h,\partial_{\circ}k)$. 
\end{defn}

\begin{defn}
    Let $\cD=(Q,K)$ be a datum. 
    The set of \emph{additive coboundaries} (resp. \emph{multiplicative coboundaries})
    is 
    $A^2(\cD)$ (resp. $W^2(\cD)$) is the subset of $Z^2_{b}(\cD)$
    that consists of all the pairs 
    $(\partial_+h,\mu)$ (resp. $(\alpha,\partial_{\circ}k)$).
\end{defn}

The following result is straightforward. 

\begin{pro}
\label{pro:K}
    Let $\cD=(Q,K)$ be a datum. 
    Then there exist canonical 
    group homomorphisms 
    \begin{align*}
    \delta &\colon H_b^2(\cD)\to H^2(\cD_+)\times H^2(\cD_\circ),\quad 
        \overline{(\alpha,\mu)}\mapsto (\overline{\alpha},\overline{\mu}),\\
    \delta_+ &\colon H_b^2(\cD)\to H^2(\cD_+),\quad 
        \overline{(\alpha,\mu)}\mapsto (\overline{\alpha}),\\
    \delta_{\circ} &\colon H_b^2(\cD)\to H^2(\cD_\circ),\quad 
        \overline{(\alpha,\mu)}\mapsto (\overline{\mu}).
    \end{align*}
\end{pro}

\begin{notation}
    We write
    \begin{align*}
        \ker\delta &= S^2(\cD) \simeq P^2(\cD)/B_{b}(\cD),\\
        \ker\delta_+ &= S^2_+(\cD) \simeq A^2(\cD)/B_{b}(\cD),\\
        \ker\delta_\circ &= S^2_\circ(\cD) \simeq W^2(\cD)/B_{b}(\cD).
    \end{align*}
\end{notation}

\begin{notation}
\label{not:sequence}
    Let $\cD=(Q,K)$ be a datum. We write 
    $\sigma(\cD,\iota,\delta)$ 
    to denote the sequence 
    \[
    \begin{tikzcd}
        0\arrow[r] & S^2(\cD)\arrow[r,"\iota"] & H^2_{b}(\cD)\arrow[r,"\delta"] & H^2(\cD_+)\times H^2(\cD_\circ)\arrow[r] & 0,
    \end{tikzcd}
    \]
    of groups and homomorphisms, where 
    $\cD_+=(Q_+,K)$ and 
    $\cD_\circ=(Q_\circ,K)$ and
    $\delta$ is the map of Proposition \ref{pro:K}. 
    Similarly, 
    we write 
    $\sigma_+(\cD,\iota_+,\delta_+)$ 
    and $\sigma_{\circ}(\cD,\iota_\circ,\delta_{\circ})$ to denote the sequences
     \[
    \begin{tikzcd}
        0\arrow[r] & S_+^2(\cD)\arrow[r,"\iota"] & H^2_{b}(\cD)\arrow[r,"\delta_+"] & H^2(\cD_+)\arrow[r] & 0,
    \end{tikzcd}
    \]
    \[
    \begin{tikzcd}
        0\arrow[r] & S_{\circ}^2(\cD)\arrow[r,"\iota"] & H^2_{b}(\cD)\arrow[r,"\delta_{\circ}"] & H^2(\cD_\circ)\arrow[r] & 0,
    \end{tikzcd}
    \]
    respectively, where $\delta_+$ and $\delta_{\circ}$ are the maps of Proposition \ref{pro:K}. 
    To simplify the notation sometimes we will write for example $\sigma(\cD)=\sigma(\cD,\iota,\delta)$. 
\end{notation}

\begin{rem}
    The sequence $\sigma(\cD)$ is exact at $S^2(\cD)$ and $H^2_{b}(\cD)$ but it might 
    not be exact at $H^2(\cD_+)\times H^2(\cD_\circ)$. In Lemmas \ref{lem:section} and \ref{lem:unlabeled}, Theorems \ref{thm:Schur_cyclic_abelian},  \ref{thm:Mb(NxT)} and \ref{thm:directproductbicyclic}, we will present examples 
where some of the sequences of Notation \ref{not:sequence} are exact. 
\end{rem}

\begin{pro}
\label{pro:bilinearfactorset}
Let $D=(Q,K)$ be a datum and $x,y,z\in Q$, then an element of $S_+(\cD)$ has a factor set representative $(1,f)$ such that 
\begin{equation}
\label{eq:rightlin}
    f(x,y+z)= f(x,y)f(x,z)
\end{equation}
In addition, if $y+z=y\circ z$, then
\begin{equation}\label{eq:leftlin}
f(x\circ y,z)= f(x,z)f(y,z)    
\end{equation}
\end{pro}
\begin{proof}
The elements of $S_+(\cD)$ are the equivalence classes of 
factor sets $(\partial_+h,\mu)$ for some $h\colon Q\to K$.
Note that every $(\partial_+h,\mu)\in\ker(\delta_+)$ 
is cohomologous to $(1,\partial_{\circ}h^{-1}\mu)$.
    Equality \eqref{eq:rightlin} is a reformulation of \eqref{eq:FS5}.
Equality \eqref{eq:leftlin} is a direct consequence of 
\eqref{eq:FS4} in the case $y+z=y\circ z$.
\end{proof}


\begin{lem}
\label{lem:section}
    Let $G$ be a trivial skew brace, $K$ an abelian group, and $\cD=(G,K)$. The sequence $\sigma_+(\cD,\iota,\delta)$ splits.
\end{lem}

\begin{proof}
Consider the group homomorphism 
$Z^2(\cD_+)\to H^2_{b}(\cD)$, 
$f\mapsto\overline{(f,f)}$,
where $Z^2(\cD)$ is the group of group 2-cocycles. 
We need to verify that this map is well-defined. 
We claim that in the case of trivial skew braces, \eqref{eq:FS5} becomes
\begin{equation}
\label{eq:trivial_FS5}
\begin{aligned}
\alpha(x,y)\alpha(x,y+z)^{-1}\alpha(x,z)=\mu(x,y)\mu(x,y+z)^{-1}\mu(x,z).
\end{aligned}
\end{equation}
In fact, Equality \eqref{eq:FS3} yields 
\begin{align*}
\alpha(x+y,-x)^{-1}\alpha(x+y-x,x+z)^{-1}&=\alpha(-x,x+z)^{-1}\alpha(x+y,z)^{-1},\\
\alpha(y,z)\alpha(x+y,z)^{-1}&=\alpha(x,y)\alpha(x,y+z)^{-1},\\
\alpha(-x,x)\alpha(-x,x+z)^{-1}&=\alpha(x,z). 
\end{align*}
A straightforward calculation with 
these formulas and Equality \eqref{eq:FS5} 
yields the claim.
It is straightforward that our map is trivial on the coboundaries.
Thus it factorizes through a map $s\colon H^2(\cD_+)\to H^2_{b}(\cD)$.
\end{proof}

\begin{thm}
\label{thm:M(Gop)}
    Let $\cD=(G,K)$ be a datum, where $G$ is a trivial skew brace. 
    Then $H_{b}^2(\cD)\simeq H^2(\cD)\times \Hom(G,\Hom(G,K))$.
\end{thm}

\begin{proof}
By Lemma \ref{lem:section} it is enough to show that $S_+^2(\cD)\simeq \Hom(G,\Hom(G,K))$. 
By Proposition \ref{pro:bilinearfactorset}, the elements of $S_+^2(\cD)$ are of the form 
$\overline{(1,f)}$ for some $f$ with 
\begin{equation}
\label{eq:f_bimorphism}
f(x,y+z)=f(x,y)f(x,z)\quad \text{and} \quad f(x+y, z)= f(x,z)f(y,z).
\end{equation}
This implies that $f\in \Hom(G,\Hom(G,K))$.

We claim that the map 
\[
\Hom(G,\Hom(G,K))\to S_+(\cD),\quad 
f\mapsto \overline{(1,f)},
\]
is an isomorphism. 

The map is a well-defined surjective group homomorphism. It remains to prove 
that the map is injective, so we need to show that $\overline{(1,f)}\in S_+(\cD)$ 
is uniquely determined by $f$. That is, two factor sets $(1,f_1)$ and $(1,f_2)$ are cohomologous 
if and only if $f_1=f_2$. Assume $(1,f_1)$ and $(1,f_2)$ are cohomologous. 
Then there exists a map $h\colon G\to K$ with $h(0)=1$ such that $(1,f_1)=(\partial_+h,\partial_{\circ}hf_2)$. Since $\partial_+ h =1$, it follows that $\partial_{\circ} h=1$. 
\end{proof}



\begin{thm}
\label{thm:G=Gop}
Let $G$ be a group. Then $H_b^2(G,K)\simeq H_b^2(G^{\op},K)$.
\end{thm}

\begin{proof}
We claim that with datum $(G^{\op},K)$, Equality~\eqref{eq:FS5} becomes
\[
\alpha(y,x)\mu(x,y)^{-1}\alpha(y,z+x)^{-1}\mu(z+x,y)\alpha(y,z)\mu(z,y)^{-1}=1.
\]
We have
\begin{gather}
    \label{eq:FS5_op1}
    \alpha(x,-x)\alpha(y+x,-x)^{-1}=\alpha(y,x),\\
    \label{eq:FS5_op2}
    \mu(x,z\circ y)\mu(x,z)^{-1}
    =\mu(z,y)^{-1}\mu(z+x,y).
\end{gather}
By using \eqref{eq:FS5_op1} and \eqref{eq:FS5_op2} 
on \eqref{eq:FS5} we obtain the claim.
Let 
\[
\psi\colon Z^2_{b}(G,K)\to H^2_{b}(G^{\op},K),\quad 
(\alpha,\mu)\mapsto \overline{(\alpha,\mu\tau)},
\]
where $\tau$ denotes the flip map.  
The map $\psi$ is well-defined. Moreover, $\psi$  
factors through a homomorphism $\phi\colon H_b^2(G,K)\to H_b^2(G^{op},K)$. Similar
arguments prove 
that the map 
$H^2_{b}(G^{\op},K)\to H^2_{b}(G,K)$, 
$\overline{(\alpha,\mu)}\mapsto \overline{(\alpha,\mu\tau)}$, 
is well-defined and it is the inverse of $\psi$. 
\end{proof}

\subsection{Hoschild--Serre exact sequence for skew braces}

\begin{notation}
    Let $K_1,\dots ,K_n$ be abelian groups and $Q$ be a skew brace. For 
    $i\in\{1,\dots,n\}$ let $g_i\colon Q\to K_i$ be maps. We write 
    \[
    (g_1,\dots,g_n)\colon Q\to K_1\times\cdots\times K_n,
    \quad
    (g_1,\dots,g_n)(x)=(g_1(x),\dots,g_n(x)).
    \]
\end{notation}

\begin{pro}
    Let $K_1,\dots ,K_n$ be finite abelian groups and $Q$ be a finite skew brace.
    The map 
    \begin{align*}
    \theta\colon \prod_{i=1}^n H^2_{b}\left(Q,K_i\right)&\to H^2_{b}\left(Q,K_1\times\cdots\times K_n\right),\\
    \left(\overline{(\alpha_1,\mu_1)},\dots,\overline{(\alpha_n,\mu_n)}\right)&\mapsto\overline{\left((\alpha_1,\dots,\alpha_n),(\mu_1,\dots,\mu_n)\right)},
    \end{align*}
    is a group isomorphism. 
\end{pro}

\begin{proof}
Let $K=K_1\times\cdots\times K_n$ and 
\begin{align*}
\gamma\colon\prod_{i=1}^n Z^2_b(Q,K_i)&\to H^2_{b}(Q,K),\\
\left((\alpha_1,\mu_1),\dots,(\alpha_n,\mu_n)\right)&\mapsto\overline{\left((\alpha_1,\dots,\alpha_n),(\mu_1,\dots,\mu_n)\right)}.
\end{align*}
Let 
\[
\left((\alpha_1,\mu_1),\dots,(\alpha_n,\mu_n)\right)\in \prod_{i=1}^n Z^2_b(Q,K_i).
\]
A straightforward computation shows that 
$(\alpha_1,\dots,\alpha_n)$ and $(\mu_1,\dots,\mu_n)$ satisfy \eqref{eq:FS1}--\eqref{eq:FS5}. Thus $\gamma$ is well-defined. Moreover, 
$\gamma$ is a group homomorphism. 
Let 
\[
f=\left(\left(\partial_+h_1,\partial_{\circ} h_1\right),\dots , \left(\partial_+h_n,\partial_{\circ}h_n\right)\right)\in \prod_{i=1}^n B_b(Q,K_i)
\]
and $h\colon Q\to K$, $h(x)=(h_1(x),\dots,h_n(x))$. Since 
\begin{align*}
    \gamma\left(f\right)&=\overline{\left((\partial_+h_1,\dots,\partial_+h_n),(\partial_{\circ}h_1,\dots,\partial_{\circ}h_n)\right)}
    =\overline{(\partial_+h,\partial_\circ h)}=1,
\end{align*}
$\gamma$ factors through $\theta$. It is left to show that $\theta$ is bijective. 
Let 
\[
\beta\colon Z^2_b(Q,K)\to \prod_{i=1}^n H^2_{b}(Q,K_i),
\quad
(\alpha,\mu)\mapsto \left(\overline{\left(p_1\alpha,p_1\mu\right)},\dots,\overline{\left(p_n\alpha,p_n\mu\right)}\right),
\]
where each $p_i\colon K\rightarrow K_i$ denotes the $i$-th projection. 
As we did before, routine calculations show that $\beta$ is well-defined, factors through 
a map 
\[H^2_b(Q,K)\to \prod_{i=1}^n H^2_{b}(Q,K_i)
\]
and that it is the inverse of $\theta$. 
\end{proof}
Let 
\begin{equation}
\label{eq:ann_extension}
\begin{tikzcd}
0\arrow[r] & K \arrow[r,"\iota"] & B\arrow[r,"\pi"] & Q\arrow[r] & 0
\end{tikzcd}
\end{equation}
be an annihilator extension of skew braces and let $A$ be an abelian group. 
A skew brace homomorphism $f\in\Hom(B,A)$ induces a natural group homomorphism 
$g\in\Hom(K,A)$ by restricting $f$ to $K$. The assignment $f\mapsto g$ yields a group homomorphism 
\[
\Res\colon \Hom(B,A)\to\Hom(K,A).
\]
We call the map $\Res$ the \emph{restriction homomorphism}. 

Let $(\alpha,\mu)$ be a factor set of the extension \eqref{eq:ann_extension}. For every group homomorphism $\phi\in\Hom(K,A)$ it follows that 
$(\phi\alpha,\phi\mu)\in Z^2_b(Q,A)$. Routine calculations show that the assignment $\phi\mapsto\overline{(\phi\alpha,\phi\mu)}$ 
does not depend on $(\alpha,\mu)$ and then 
it gives a group homomorphism 
\[
\Tra\colon\Hom(K,A)\to H_{b}^2(Q,A).
\]
We call the map $\Tra$ the \emph{transgression homomorphism}. 

The map $\Inf\colon\Hom(Q,A)\to\Hom(B,A)$, $\psi\mapsto\psi\pi$, is a group homomorphism. 
Routine calculations show that the map
\[
\Inf\colon H_{b}^2(Q,A)\to H_{b}^2(B,A),\quad
\overline{(\alpha_1,\mu)}\mapsto\overline{(\alpha_1(\pi\times\pi),\mu_1(\pi\times\pi))},
\]
is a well-defined group homomorphism. We call these maps $\Inf$ the \emph{inflation homomorphisms}. 

\begin{thm}
\label{thm:HoschildSerre}
Let 
\begin{equation}
\label{eq:HS_myextension}
\begin{tikzcd}
0\arrow[r] & K \arrow[r,"\iota"] & B\arrow[r,"\pi"] & Q\arrow[r] & 0
\end{tikzcd}
\end{equation}
be an annihilator extension of skew braces. Let $A$ be an abelian group. Then
\[
\begin{tikzcd}
	0 & {\Hom(Q,A)} & {\Hom(B,A)} & {\Hom(K,A)} \\
	& {H^2_b(Q,A)} & {H_b^2(B,A)} & {}
	\arrow[from=1-1, to=1-2]
	\arrow["\Inf", from=1-2, to=1-3]
	\arrow["\Res", from=1-3, to=1-4]
	\arrow["\Inf"', from=2-2, to=2-3]
	\arrow["\Tra"{description}, curve={height=6pt}, from=1-4, to=2-2]
\end{tikzcd}
\]
is an exact sequence. 
\end{thm}

\begin{proof}
We identify $K$ with $\iota(K)$. 
The map $\pi$ is surjective, so the 
sequence is exact at $\Hom(Q,A)$. 

Clearly, the sequence is exact at $\Hom(B,A)$. 

We now prove that the sequence is exact at $\Hom(K,A)$. 
Let $(\alpha,\mu)$ be a factor set of \eqref{eq:HS_myextension} and let $t$ be a transversal of \eqref{eq:HS_myextension}. We identify 
$B=K\times_{(\alpha,\mu)}Q$ (see Theorem \ref{thm:bijection}). 
Let $g\in\Hom(B,A)$. Then 
\[
\Tra(\Res(g))=\overline{\left(\Res(g)\alpha,\Res(g)\mu\right)}.
\]
Since both 
\begin{align*}
&\Res(g)\alpha(x,y) = g\left(t(x)+t(y)-t(x+y)\right)=g\left(t(y)\right)g\left(t(x+y)\right)^{-1}g\left(t(x)\right),\\
&\Res(g)\mu(x,y) = g\left(t\left(x\right)\circ t\left(y\right)\circ t\left(x\circ y\right)'\right)=g\left(t\left(y\right)\right)g\left(t\left(x\circ y\right)\right)^{-1}g\left(t\left(x\right)\right), 
\end{align*}
hold for all $x,y\in Q$, 
the pair $\left(\Res(g)\alpha,\Res(g)\mu\right)$ is a coboundary and hence 
$\Tra\Res$ is trivial. 

Now let $g\in\Ker(\Tra)$. Our goal is to extend $g$ to $\Hom(B,A)$. 
By hypothesis, there exists a map $h\colon Q\to K$ such that
\[
g\left(\alpha\left(x,y\right)\right)=h\left(y\right)h\left(x+y\right)^{-1}h\left(x\right),
\quad
g\left(\mu\left(x,y\right)\right)=h\left(y\right)h\left(x\circ y\right)^{-1}h\left(x\right),
\]
for all $x,y\in Q$. 
For $k\in K$ and $x\in Q$, we define 
\[
g'\colon B\to A,
\quad
k+t\left(x\right)\mapsto g\left(k\right)h\left(x\right).
\]
The map is well defined by Remark \ref{rem:decomposition}. 

Now $g'$ is a skew brace homomorphism, as 
\begin{align*}
g'\left(a+t\left(x\right)+b+t\left(y\right)\right)&=g'\left(a + b + \alpha\left(x,y\right) + t\left(x+y\right)\right)\\
&= g\left(a\right) h\left(x\right) g\left(b\right)h\left(y\right)
\end{align*}
and similarly 
\[
g'\left(a\circ t\left(x\right)\circ b\circ t\left(y\right)\right)
=g\left(a\right) h\left(x\right) g\left(b\right)h\left(y\right).
\]
In addition, $g'$ restricts to $g$ on $K$. 
Therefore, $\Ker(\Tra) =\Img\Res$. 

It is left to check that the sequence is exact at $H_{b}^2(Q,A)$. First we show that $\Inf\Tra$ is trivial. 
Let $\phi\in\Hom(K,A)$.  
Write
\[
\Inf\Tra(\phi)=\overline{(\phi\alpha',\phi\mu')},
\]
where the maps $\alpha',\mu'\colon B\times B\to K$ are such that $\alpha'(x,y)=\alpha(\pi(x),\pi(y))$ 
and $\mu'(x,y)=\mu(\pi(x),\pi(y))$. 
Let $h\colon B\to A$, $x\mapsto \phi\left( t\left(\pi\left(x\right)\right)-x\right)$. One checks that $t(\pi(x))-x\in\ker\pi$, so 
$h$ is well-defined. Notice that 
\begin{equation}
    \label{eq:trick_exact}
    \begin{aligned}
        t\left(\pi\left(x\right)\right)\circ x'= & t\left(\pi\left(x\right)\right)\circ x'+ x-x\\
        =& t\left(\pi\left(x\right)\right)\circ x'\circ x-x
        =  t\left(\pi\left(x\right)\right)-x.
    \end{aligned}
\end{equation}
Then
\begin{align*}
    \phi\alpha'\left(x,y\right) & = \phi\left(\alpha\left(\pi\left(x\right),\pi\left(y\right)\right)\right)\\
    &= \phi\left(t\left(\pi\left(x\right)\right)+t\left(\pi\left(y\right)\right) -t\left(\pi\left(x+y\right)\right)\right)\\
    &= \phi\left(t\left(\pi\left(x\right)\right) -x+t\left(\pi\left(y\right)\right)-y -\left(t\left(\pi\left(x+y\right)\right)-y-x\right)\right)\\
    &= \phi\left(t\left(\pi\left(x\right)\right) -x\right)\phi\left(t\left(\pi\left(y\right)\right)-y\right) \phi\left(t\left(\pi\left(x+y\right)\right)-y-x\right)^{-1}\\
    &= \partial_+h\left(x,y\right).
\end{align*}
Similarly, by \eqref{eq:trick_exact}, 
$\phi\mu'(x,y)=\partial_{\circ}h\left(x,y\right)$.
Therefore $\Inf\Tra$ is trivial. 

Finally, assume that $\overline{(\alpha_1,\mu_1)}\in \Ker(\Inf)$. Then there exists a map $h\colon B\to A$ such that for all $a,b\in K$ and $x,y\in Q$,
\begin{equation}
\label{eq:KerInfImTra}
\begin{aligned}
\alpha_1(x,y)&= h(b+t(y))h(a+b+\alpha(x,y)+t(x+y))^{-1}h(a+t(x)),\\
    \mu_1(x,y)&=h(b+t(y))h(a+b+\mu(x,y)+t(x\circ y))^{-1}h(a+t(x)).
\end{aligned}
\end{equation}
Consider the maps $\phi\colon K\to A$, $k\to h(k)$, and $p\colon Q\to A$, $x\to h(t(x))$. By setting $x=y=0$ (resp. $y=0$ and $a=0$) in \eqref{eq:KerInfImTra} we see that $\phi$ is in fact a group homomorphism (resp. $h(b+t(x))=\phi(b)p(x)$). Hence, setting $a=b=0$ in \eqref{eq:KerInfImTra}, we obtain that 
\begin{gather*}
\alpha_1(x,y)=p(y)p(x+ y)^{-1}p(x)\phi(\alpha(x,y))^{-1},\\
\mu_1(x,y)=p(y)p(x+ y)^{-1}p(x)\phi(\mu(x,y))^{-1}.
\end{gather*}
This means that $\overline{(\alpha_1,\mu_1)}\in \Img(\Tra)$.
\end{proof}

\section{The Schur multiplier}
\label{Schur}

\begin{defn}
\label{defn:SchurMultiplier}
    The \emph{Schur multiplier} of a finite skew brace $Q$ is 
    defined as the group $M_b(Q)=H^2_b(Q,\C^{\times})$.
\end{defn}

The Schur multiplier of a finite skew 
brace is always an abelian group.

\begin{thm}
\label{thm:M(Q)exp}
    If $Q$ is a finite skew brace, 
    then $\exp(M_b(Q))$ divides $|Q|^2$.
\end{thm}

\begin{proof}
    Let $n=|Q|$. Let 
    \begin{equation}
    \label{eq:AandM}
    A(x)=\prod_{z\in Q}\alpha(x,z),
    \quad
    M(x)=\prod_{z\in Q}\mu(x,z).
    \end{equation}
    Note that $\alpha(y,z)\alpha(x+y,z)^{-1}\alpha(x,y+z)=\alpha(x,y)$ for all $x,y,z\in Q$.  
    By multiplying this equality for every $x\in Q$, 
    \begin{equation}
        \label{eq:A}
        A(y)A(x+y)^{-1}A(x)=\alpha(x,y)^n.
    \end{equation}
    Similarly, 
    \begin{equation}
        \label{eq:M}
        M(y)M(x\circ y)^{-1}M(x)=\mu(x,y)^n.
    \end{equation}
    By using \eqref{eq:FS5} we obtain that
    \begin{equation}
    \label{eq:AM}
    A(y)M(x)\mu(x,y)^{-n}\alpha(x,-x)^n\alpha(x\circ y,-x)^{-n}M(x)^{-1}=A(x\circ y-x).
    \end{equation}
    Combining \eqref{eq:A}, \eqref{eq:M} and \eqref{eq:AM} and canceling some factors, 
    \[
    A(y)A(x\circ y)^{-1}A(x)=M(y)M(x\circ y)^{-1}M(x).
    \]
    It follows that
    \begin{align*}
    M(x)^n&=\prod_{z\in Q}\mu(x,z)^n\\
    &=\prod_{z\in Q}M(z)M(x\circ z)^{-1}M(x)\\
    &=\prod_{z\in Q}A(z)A(x\circ z)^{-1}A(x)=A(x)^n.
    \end{align*}
    Let $h\colon Q\to\C^{\times}$, $h(x)=M(x)^n$. Then $h(0)=1$ and
    \[
    h(y)h(x+y)^{-1}h(x)=\alpha(x,y)^{n^2},
    \quad
    h(y)h(x\circ y)h(x)=\mu(x,y)^{n^2}.
    \]
    Hence $M_b(Q)^{n^2}=0$. 
\end{proof}

For $n\in\Z_{>1}$ we write $G_n=\{z\in\C:z^n=1\}$.

\begin{thm}
\label{thm:M(Q)finite}
    If $Q$ is a finite skew brace, then $M_b(Q)$ is finite. More precisely, 
    every element of $M_b(Q)$ of order $n$ 
    has a representative $(\alpha,\mu)$ such that $\Img(\alpha),\Img(\mu)\subseteq G_{n}$.
\end{thm}

\begin{proof}
    Let $(\alpha,\mu)$ be a factor set such that there exists a map $k:Q\to \C^{\times}$ with $k(0)=1$ and $(\alpha^n,\mu^n)=(\partial_+ k,\partial_\circ k)$.
    Let $h(0)=1$ and 
    for $x\in Q\setminus\{0\}$ let $h(x)$ be an $n$-th root of $k(x)^{-1}$. Define the maps $\alpha',\mu'\colon Q\times Q \to \C^{\times}$ such that 
    \begin{align*}
        \alpha'(x,y)&=\alpha(x,y) h(y)h(x+y)^{-1} h(x),\\
        \mu'(x,y)&=\mu(x,y) h(y)h(x\circ y)^{-1}h(x).    
    \end{align*} 
    Clearly, $(\alpha',\mu')$ is cohomologous to $(\alpha,\mu)$. Moreover,
    \[
    \alpha'(x,y)^{n}=\mu(x,y)^{n}=1.
    \]
   Therefore, by Theorem \ref{thm:M(Q)exp} each element of $M_b(Q)$ is uniquely determined by a map $Q\times Q\rightarrow G_{n^2}\times G_{n^2}$. There is only a finite number of such maps; thus $M_b(Q)$ is finite.
\end{proof}


\subsection{Trivial and almost trivial skew braces}



\begin{pro}
Let $G$ be a group. Then 
$M_{b}(G)\simeq M_{b}(G^{\op})\simeq M(G)\times (G'\otimes G')$.
\end{pro}

\begin{proof}
    It follows from Theorems \ref{thm:M(Gop)} and \ref{thm:G=Gop}.
\end{proof}

\begin{exa}
For positive integers $n_1,\dots,n_k$ such that 
$n_{i+1}$ divides $n_i$ for all $1\leq i\leq k-1$ let 
$G$ be the abelian group 
\[
G=\Z/{(n_1)}\times\Z/{(n_2)}\times\dots\times\Z/{(n_k)}.
\]
By Theorem \ref{thm:directproduct}, 
$M(G)\simeq \Z/{(n_2)}\times \Z/(n_3)^2\times\dots \times \Z/{(n_k)}^{k-1}$. 
Since the Schur multiplier of a finite cyclic group
is trivial (see for example \cite[Proposition 2.1.1]{MR1200015}), 
Theorem \ref{thm:M(Gop)} implies 
\[
M_{b}(G)\simeq \Z/{(n_1)}\times\Z/{(n_2)}^4\times\Z/{(n_3)}^7\times\dots \times \Z/{(n_k)}^{3k-2}.
\]
\end{exa}



\begin{exa}
    If $Q_8$ denotes the quaternion group, then $M(Q_8)$ is trivial. Thus
    $M_{b}(Q_8)\simeq \Z/{(2)}^4$.
\end{exa}

\subsection{Cyclic skew braces}


\begin{defn}
    A skew brace $Q$ is called cyclic if $Q_+$ is cyclic and \emph{bicyclic} if both groups $Q_+$ and $Q_\circ$ are cyclic.
\end{defn}

Cyclic skew braces with abelian multiplicative group
were classified by Rump \cite{MR2298848}. 
The following lemma goes back to Rump. 

\begin{lem}
Every skew brace with additive group isomorphic to $\Z/(n)$ and abelian multiplicative group  
is a skew brace isomorphic 
to $C_{(n,d)}$ for some $d$ dividing $n$, see Example \ref{exa:A(n,d)}. 
\end{lem} 

\begin{pro}[Proposition 6 of \cite{MR2298848}]
\label{pro:Rump}
    Let $n\geq2$ and $d$ be a divisor of $n$. 
    Then $C_{(n,d)}$ 
    is bicyclic if and only if 
    $4\mid d$ whenever $4\mid n$. In the exceptional case where $4\mid n$ and $2$ is the highest power of $2$ dividing $d$, the multiplicative 
    group of $C_{(n,d)}$ is isomorphic to 
    $\langle 1\rangle\times\langle -1\rangle\simeq \Z/{\left(\frac{n}{2}\right)}\times\Z/{(2)}$.
\end{pro}

\begin{rem}
    We can define the operation $a\circ b= a+2ab+b$ for all integers $a,b$.
 One proves by induction on $k$ that
    $1^{\circ k}=\frac{3^k-1}{2}$ for all $k\geq1$. 
\end{rem}

\begin{rem}
\label{rem:divisibility}
Let $n\geq 1$, then $2^{n+2}\mid 3^{2^n}-1$ and $2^{n+3}\nmid 3^{2^n}-1$, as 
\[
3^{2^n}-1= (3-1)(3+1)(3^2+1)(3^{2^2}+1)\dots(3^{2^{n-1}}+1).
\]
\end{rem}
\begin{pro}
\label{pro:generatorprecoboundaries}
    Let $n\geq2$ and $d$ be a divisor of $n$ such that 
    $C_{(n,d)}$ is bicylcic or $n=2^{m+1}$ and $d=2$ for some $m\geq 1$. Let $\xi$ be an $n$-th primitive root of unity and let $\gamma\colon C_{(n,d)}\times C_{(n,d)}\to \C^{\times}$, 
$(x,y)\mapsto \xi^{xy}$. Then $(1,\gamma)\in P(C_{(n,d)},\C^{\times})$.
\end{pro}
\begin{proof}
Direct calculations 
show that 
\begin{gather*}
    \gamma(y,z)\gamma(x\circ y,z)^{-1}\gamma(x,y\circ z)\gamma(x,y)^{-1}=1,\quad 
    \gamma(x,y+z)=\gamma(x,y)\gamma(x,z)
\end{gather*}
for all $x,y,z\in C_{(n,d)}$.

If $C_{(n,d)}$ is bicyclic, then $\gamma$ is a coboundary, as the Schur multiplier of the multiplicative group is trivial. Hence the claim follows.
Otherwise, let $B=C_{(2^{m+1},2)}$ and 
$\mu$ be a $2^{m+2}$-th primitive root of one such that $f(x,y)= \mu^{2xy}$.
Let 
\[
h\colon B\to \C^{\times},
\quad 
1^{\circ k}\circ(-1)^{\circ s}\mapsto \mu^{1^{\circ k}\circ (-1)^{\circ s}+2k}.
\]
This makes sense because of Proposition \ref{pro:Rump}. The definition does not depend on the residue of $s$ modulo $2$.
We need to verify that it does not depend on the residue of $k$ modulo $2^m$. We claim that $\forall a\in \Z$
\[
1^{\circ (k+a2^m)}\circ(-1)^{\circ s}+2(k+a2^m)\equiv 1^{\circ k}\circ (-1)^{\circ s}+2k\bmod 2^{m+2}.
\]
As $1^{\circ 2^{m+1}}\equiv 0 \bmod 2^{m+2}$ it is enough to check in case $a=1$. By Remark \ref{rem:divisibility}, there exists an odd integer $\alpha$ such that $1^{\circ 2^m}=\alpha 2^{m+1}$. Thus  
\[
1^{\circ 2^m}+2^{m+1}= 2^{m+1}(\alpha+1)\equiv  0\bmod 2^{m+2}.
\]
which proves the claim. Finally,
direct computations show that $\partial_{\circ}h= f$. 
\end{proof}
\begin{thm}
\label{thm:precoboundariescycl}
    Let $n\geq 2$ and $d$ be as in Proposition \ref{pro:generatorprecoboundaries}, $S(C_{(n,d)},\C^{\times})\simeq \Z/(d)$.
\end{thm}
\begin{proof}
    Write $Q=C_{(n,d)}$. Let $\xi$ and $\gamma$ be as in Proposition \ref{pro:generatorprecoboundaries}.
We claim that 
$\overline{(1,\gamma)}$
generates $S(C_{(n,d)},\C^{\times})$.
Every element of $S(C_{(n,d)},\C^{\times})$ has a representative of the form
$(1,\mu)$ for some $\mu\colon Q\times Q\to\C^{\times}$. 
In addition, 
\begin{equation}
\label{eq:mu}
\mu(x,y+z)=\mu(x,y)\mu(x,z)
\end{equation}
holds for all $x,y,z\in Q$ 
and $\mu=\partial_{\circ} h$ for some $h\colon Q\to\C^{\times}$. Thus 
\begin{equation}
    \label{eq:muxy=muyx}
    \mu(x,y)=h(y)h(x\circ y)^{-1}h(x)=h(x)h(y\circ x)^{-1}h(y)=\mu(y,x)
\end{equation}
for all $x,y\in Q$. From Equalities \eqref{eq:mu} 
and \eqref{eq:muxy=muyx}, 
it follows that 
\[
\mu(x,y)^n=\mu(x,ny)=1
\]
and 
that $\mu$ is uniquely determined by $\mu(1,1)$, as 
\[
\mu(x,y)=\mu(x,1)^y=\mu(1,x)^y=\mu(1,1)^{xy}.
\]
As $\mu(1,1)=\xi^k$ for some $k$,
$\overline{(1,\gamma)}$ generates $S(C_{(n,d)},\C^{\times})$. 

We claim that $d$ is the order of $\overline{(1,\gamma)}$. Let 
$q$ be the order of $\overline{(1,\gamma)}$. Then
\[
(1,\gamma^q)=(\partial_+ k,\partial_{\circ} k)
\]
for some $k\colon Q\to\C^{\times}$ such that $k(0)=1$. It follows that
$k(x+y)=k(x)k(y)$ for all $x,y\in Q$. Since
$Q_+$ is cyclic, 
$k(x)=k(1)^x$ for all $x\in Q$. In particular, since 
\[
1=k(0)=k(n)=k(1)^n,
\]
it follows that $k(1)=\xi^l$ for some $l$. Note that
\[
\xi^q=\gamma(1,1)^q= k(1)^{2}k(1\circ 1)^{-1}=\xi^{2l}\xi^{-l(2+d)}=\xi^{-ld}.
\]
Thus $n$ divides $q+ld$. Since $d$ divides $n$, it follows that $d$ divides $q$. 
Let 
\[
s\colon Q\to\C^{\times},\quad 
x\mapsto\xi^{-x}.
\]
A direct calculation shows that 
$\partial_+s(x,y)=1$ 
and $\partial_\circ s(x,y)=\gamma(x,y)^d$
for all $x,y\in Q$.
\end{proof}
\begin{cor}
    \label{cor:Schur_bicyclic}
    Let $n\geq2$ and $d$ be a divisor of $n$ such that 
    $C_{(n,d)}$ is bicyclic. 
    Then $M_b(C_{(n,d)})\simeq\Z/(d)$. 
\end{cor}
\begin{proof}
    This is a direct consequence of Theorem \ref{thm:precoboundariescycl} and the fact that the Schur multiplier of a cyclic group is trivial.
\end{proof}

\begin{thm}
\label{thm:Schur_cyclic_abelian}
Let $m\geq1$
$M_{b}(C_{(2^{m+1},2)})\simeq \Z/(2)\times\Z/(2)$.
\end{thm}

\begin{proof}
    Let $B=C_{(2^{m+1},2)}$. It is enough to show that the sequence
    \[
    \begin{tikzcd}
        0\arrow[r] & S(B,\C^{\times})\arrow[r,"\iota"] & M_{b}(B)\arrow[r,"\delta"] &M(B_\circ)\arrow[r] & 0
    \end{tikzcd}
    \]
    is exact and splits. 
By Theorem \ref{thm:directproduct}, $M(B_\circ)=\langle\overline{\gamma}\rangle\simeq\Z/(2)$, where
\[
\gamma\colon B_\circ\times B_\circ\to\C^{\times},
\quad
(1^{\circ k}\circ(-1)^{\circ s},1^{\circ l}\circ(-1)^{\circ t})\mapsto (-1)^{sl},
\]
Let $h:B\to \C^\times; 1^{\circ k}\circ(-1)^{\circ t}\mapsto i^{(-1)^{\circ t}}$. 
As $(-1)^{\circ s}\equiv 1^{\circ s}\equiv s\: (mod\: 2)$ for all $s\in \mathbb{Z}$, straightforward computations show that 
\[\partial_\circ h\gamma(1^{\circ k}\circ(-1)^{\circ s},1^{\circ l}\circ(-1)^{\circ t})= (-1)^{(-1)^{\circ s}(1^{\circ l}\circ(-1)^{\circ t})}\]
so that the pair $(1,\partial_\circ h\gamma)$ is a factor set of skew braces. Since $(\partial_\circ h\gamma)^2=1$, 
$\overline{(1,\partial_\circ h\gamma)}$ has order two. 
Thus 
$s\colon M(B_\circ)\to M_{b}(B)$, $\overline{\gamma}\mapsto\overline{(1,\partial_\circ h\gamma})$, 
is a section of $\delta$. 
\end{proof}
\subsection{Direct product}


\begin{pro}
    Every finite radical ring with cyclic additive group and abelian non-cyclic multiplicative group 
    is isomorphic to $C_{(n,d)}\times C_{(2^{m+1},2)}$ for some $n\geq1$ odd, $d$ a divisor of $n$ and
    $m\geq1$. 
\end{pro}

\begin{proof}
    Let $A$ be a finite radical ring with cyclic additive group and abelian non-cyclic multiplicative group. 
    By \cite[Proposition 3]{MR2298848}, the skew brace $A$ admits a decomposition 
    $A=\prod_pA_p$, where
    \[
    A_p=\{a\in A:p^ma=0\text{ for some $m\geq1$}\}
    \]
    are
    ideals of $A$. By Proposition \ref{pro:Rump}, 
    $A_p$ is bycylic for every $p>2$. Thus $\prod_{p>2}A_p$ is bycylic and hence
    $\prod_{p>2}A_p=C_{(n,d)}$ for some odd integer $n$ and $d$ dividing $n$. 
    Note that $A_2$ is a radical ring with cyclic additive
    group of order a power of two. 
    Since $A$ is not bycyclic, 
    $A_2$ is not bicyclic. By Proposition \ref{pro:Rump}, 
    $A_2=C_{(2^{m+1},2)}$ for some $m\geq1$.  
\end{proof}

We will conclude the study of the Schur multiplier of cyclic radical rings by computing the
Schur multiplier of the direct product of two such rings.

\begin{lem}
\label{lem:identification}
    Let $N$ and $T$ be finite skew braces. Then 
    \[
    \Hom(N,\Hom(T,\C^{\times}))\simeq N^{\ab}\otimes T^{\ab}.
    \]
\end{lem}

\begin{proof}
    Let $\pi_T\colon T\to T^{\ab}$ be the canonical 
    skew brace homomorphism. Let $A$ be a multiplicative abelian group. 
    Routine calculations show that the map
    \[
    \Hom(T^{\ab},A)\to\Hom(T,A),
    \quad
    f\mapsto f\pi_T,
    \]
    is a well-defined bijective homomorphism of abelian groups. The surjectivity
    comes from the fact that if $g\in\Hom(T,A)$, then $g|_{T'}=1$ because,
    as $A$ is a trivial abelian skew brace, $g(t*s)=1$ and $g(t+s-t-s)=1$ for all $t,s\in T$. 
    
    Similarly, $\Hom(N^{\ab},A)\simeq\Hom(N,A)$. 
    Now 
    \begin{align*}
    \Hom(N,\Hom(T,\C^\times))&\simeq 
    \Hom(N,\Hom(T^{\ab},\C^\times))\\
    &\simeq \Hom(N^{\ab},\Hom(T^{\ab},\C^\times))
    \end{align*}
    and hence the claim follows from \cite[Lemma 2.2.9]{MR1200015}. 
\end{proof}

\begin{lem}
\label{FStrick}
Let $Q$ be a finite skew brace and $(1,f)\in Z_b(Q,\C^{\times})$. 
If $x,y,z,u\in Q$ are such that 
 \begin{align*}
     x\circ y = y\circ x,
     &&x\circ z = x+z, 
     &&y\circ u = y+u,
 \end{align*}
then
\[
f(x\circ y, z+u)= f(x,z)f(y,z)f(x,u)f(y,u).
\]
\end{lem}

\begin{proof}
    By \eqref{eq:FS5},
    $f(x\circ y, z+u)=f(x\circ y, z)f(x\circ y, u)$. 
    Using \eqref{eq:FS4} and \eqref{eq:FS5}, 
    \begin{align*}
        f(x\circ y, z)=f(y\circ x, z)&=f(x,z)f(y,x\circ z)f(y,x)^{-1}
        = f(x,z)f(y,z)
    \end{align*}
    Simillarly $f(x\circ y, u)= f(x,u)f(y,u)$.
\end{proof}

\begin{lem}
\label{lem:dirprod1}
    Let $N$ and $T$ be finite skew braces and $(1,f)\in P_{b}(N\times T,\C^{\times})$. Then 
    \[
    ((1,f_N),f_{T\times N},(1,f_T))\in P_{b}(N,\C^{\times})\times (T^{\ab}\otimes N^{\ab})\times P_{b}(T,\C^{\times}),
    \]
    where $f_N=f|_{N\times N}$, $f_T=f|_{T\times T}$ and $f_{T\times N}=f|_{T\times N}$. 
\end{lem}

\begin{proof}
Let $f=\partial_{\circ}h$. Then $f_N=\partial_{\circ}(h|_N)$ and $f_T=\partial_{\circ}(h|_T)$. In addition, 
$(1,f_N)$ and $(1,f_T)$ 
are factor set as 
$f_N$ and $f_T$ are 
restrictions of $f$. It is left to see that $f_{T\times N}\in N^{\ab}\otimes T^{\ab}$.  
Notice that
\begin{equation}
\label{eq:NxT}
h(n)f_{T\times N}(t,n)^{-1}h(t)= h(n,t) = h(n)f_{N\times T}(n,t)^{-1}h(t)
\end{equation}
where $f_{N\times T}$ denotes the restriction of $f$ to $(N\times \lbrace 0\rbrace)\times (\lbrace 0\rbrace \times T)$. In particular,
$f_{T\times N}(t,n)=f_{N\times T}(n,t)$. It follows that  
both $f_{T\times N}$ and $f_{N\times T}$  
are additive in each component.

We now show that $f_{T\times N}$ is multiplicative in both components. 
On the one hand,
using \eqref{eq:NxT} and the fact that $f=\partial_{\circ}h$, 
\begin{multline*}
    f\left((n_1,t_1),(n_2,t_2)\right)
    = f_N(n_1,n_2)f_T(t_1,t_2)
    f_{T\times N}(t_1,n_1)^{-1}\\
    f_{T\times N}(t_1\circ t_2,n_1\circ n_2)f_{T\times N}(t_2,n_2)^{-1}
\end{multline*}
On the other hand, by Lemma \ref{FStrick}, 
\begin{align*}
    f\left((n_1,t_1),(n_2,t_2)\right)= f_N(n_1,n_2)f_T(t_1,t_2)f_{T\times N}(t_1,n_2)f_{T\times N}(t_2,n_1).
\end{align*}
Therefore 
\begin{align*}
f_{T\times N}(t_1\circ t_2, n_1\circ n_2)
=f_{T\times N}(t_1,n_1)f_{T\times N}(t_2,n_2)f_{T\times N}(t_1,n_2)f_{T\times N}(t_2,n_1).
\end{align*}
This finishes the proof. 
\end{proof}

\begin{lem}
\label{lem:dirprod2}
Let $N$ and $T$ be finite skew braces.  
Let 
\[
((1,f_N),f_{N\times T},(1,f_T))\in P_{b}(N,\C^{\times})\times (T^{\ab}\otimes N^{\ab})\times P_{b}(T,\C^{\times}).
\]
Then the map $f\colon(N\times T)\times (N\times T)\to\C^{\times}$ defined by 
\[
f\left((n_1,t_1),(n_2,t_2)\right)= f_N(n_{1},n_2)f_{T\times N}(t_2,n_1)f_{T\times N}(t_1,n_2)f_T(t_1,t_2)
\]
is such that $(1,f)\in P_{b}(N\times T,\C^{\times})$. 
\end{lem}

\begin{proof}
Let $f_N=\partial_{\circ}h_N$ and $f_T= \partial_{\circ}h_T$. Let  
\[
h\colon N\times T\to \C^{\times},
\quad
(n,t)\mapsto h_N(n)f_{T\times N}(t,n)^{-1}h_T(t).
\]
A straightforward calculation shows that 
$f=\partial_{\circ}h$ and that 
$f$ is additive in the second component. 
\end{proof}

We will use the identification of Lemma \ref{lem:identification}. 

\begin{thm}
\label{thm:Kb_directproduct}
Let $N$ and $T$ be finite skew braces. Then 
\[ 
S(N\times T,\C^{\times})\simeq S(N,\C^{\times})\times (T^{\ab}\otimes N^{\ab})\times S(T,\C^{\times}).
\]
\end{thm}

\begin{proof}
Consider the maps  
\[
\phi\colon S(N,\C^{\times})\times (T^{\ab}\otimes N^{\ab})\times S(T,\C^{\times})\rightarrow  S(N\times T,\C^{\times})
\]
given by $\phi(\overline{(1,f_N)},f_{T\times N},\overline{(1,f_T)})=\overline{(1,f)}$, with $f$ as in Lemma \ref{lem:dirprod2},
and
\[
\psi\colon S(N\times T,\C^{\times})\rightarrow  S(N,\C^{\times})\times (T^{\ab}\otimes N^{\ab})\times S(T,\C^{\times})
\]
given by $\psi(\overline{(1,f)})=\left(\overline{(1,f_N)},f_{T\times N},\overline{(1,f_T)}\right)$ with maps 
$f_N$, $f_T$ and $f_{T\times N}$ as 
in Lemma \ref{lem:dirprod1}.

We first prove that $\phi$ is well-defined. Let 
\[
(\overline{(1,f_N)},f_{T\times N},\overline{(1,f_T)})=(\overline{(1,g_N)},g_{T\times N},\overline{(1,g_T)}).
\]
Then there exist $h\in\Hom(N_+, \C^{\times})$ and $k\in\Hom(T_+, \C^{\times})$ such that $g_N=f_N\partial_{\circ}h$ 
and $g_T=f_T\partial_{\circ}k$.
Let 
\[
l\colon N\times T\to \C^{\times}, (n,t)\mapsto h(n)k(t)
\]
Thus $l\in\Hom((N\times T)_+, \C^{\times})$ and
\[
    g((n_1,t_1),(n_2,t_2))= f((n_1,t_1),(n_2,t_2))\partial_{\circ}l((n_1,t_1),(n_2,t_2)).
\]

Now we prove that $\psi$ is well defined. Let $\overline{(1,f)}=\overline{(1,g)}$. Then there 
exists $l\in\Hom((N\times T)_+, \C^{\times})$ such that $g=f\partial_{\circ}l$. 
Let $h=l|_N$ and $k=l|_T$. Then $h\in\Hom(N_+, \C^{\times})$, $k\in\Hom(T_+, \C^{\times})$, $g_N=f_N\partial_{\circ}h$ and $g_T=f_T\partial_{\circ}h$. In addition, since $\partial_{\circ}l(n,t)=\partial_+l(n,t)=1$, we have $g_{T\times N}=f_{T\times N}$.

It is clear that the maps $\psi$ and $\phi$ are inverse to each other.
\end{proof}

\begin{thm}
\label{thm:Mb(NxT)}
    Let $N$ and $T$ be skew braces such that 
    $\sigma(N)$ and $\sigma(T)$ 
    split. Assume that 
    $T_+^{\ab}\otimes N_+^{\ab} = 0$ and $T_\circ^{\ab}\otimes N_\circ^{\ab} = 0$. Then 
    $\sigma(N\times T)$ splits. In particular, $M_b(N\times T) \simeq M_b(N)\times M_b(T)$.
\end{thm}

\begin{proof}
    By hypothesis and Theorem \ref{thm:directproduct},
    \[
    M((N\times T)_+)\times M((N\times T)_\circ)\simeq M(N_+)\times M(N_\circ)\times M(T_+)\times M(T_\circ).
    \]
    Let $s_1$ and $s_2$ be sections of $\sigma(N)$ and $\sigma(T)$, respectively. Let 
    \[
    (\overline{\alpha_1},\overline{\mu_1},\overline{\alpha_2},\overline{\mu_2})\in M(N_+)\times M(N_\circ)\times M(T_+)\times M(T_\circ),
    \]
    and $(\beta_1,\sigma_1)$ and 
    $(\beta_2,\sigma_2)$ be representatives of $s_1(\overline{\alpha_1},\overline{\mu_1})$ and $s_2(\overline{\alpha_2},\overline{\mu_2})$, respectively. 
    Let $(\alpha,\mu)$ be such that
    \begin{equation}
    \label{eq:alpha_mu}
    \begin{aligned}
    \alpha((n,t),(n_1,t_1))&= \beta_1(n,n_1)\beta_2(t,t_1),\\
    \mu((n,t),(n_1,t_1))&=\sigma_1(n,n_1)\sigma_2(t,t_1).
    \end{aligned}
    \end{equation}
    Straightforward computations show that $(\alpha,\mu)$ is a factor set of $N\times T$. 
    We claim that the cohomology class of $(\alpha,\mu)$ does not depend on the 
    representatives of $s_1(\overline{\alpha_1},\overline{\mu_1})$ and $s_2(\overline{\alpha_2},\overline{\mu_2})$. 
    Suppose there are maps 
    \[
    h\colon N\to \C^{\times},
    \quad
    k\colon T\to \C^{\times}
    \]
    such that $h(0)=1=k(0)$, $(\gamma_1, \eta_1)=(\beta_1\partial_+h,\sigma_1\partial_{\circ}h)$ and $(\gamma_2, \eta_2)=(\beta_2\partial_+k,\sigma_2\partial_{\circ}k)$. Then the map 
    \[
    l\colon N\times T\to \C^{\times},\quad l(n,t)=h(n)k(t)
    \]
    is such that the pair of maps built from $(\gamma_1, \eta_1)$ and $(\gamma_2, \eta_2)$ 
    as in Equalities \eqref{eq:alpha_mu} is $(\alpha\partial_+l,\mu\partial_{\circ}l)$. This yields a map 
    \begin{align*}
    s\colon M(N_+)\times M(N_\circ)\times M(T_+)\times M(T_\circ)&\to M_b(N\times T),\\
    (\overline{\alpha_1},\overline{\mu_1},\overline{\alpha_2},\overline{\mu_2})&\mapsto\overline{(\alpha,\mu)}
    \end{align*}
    One can see that, by definition, $\alpha$ and $\mu$ are normal factor set of groups. 
    Consider the sequence $S(N\times T,\iota,\delta)$. 
    The identification of Theorem \ref{thm:directproduct} shows that 
    \begin{align*}
    \delta\left(\overline{(\alpha,\mu)}\right)
    =\left(\overline{\alpha|_{N\times N}},\overline{\mu|_{N\times N}},\overline{\alpha|_{T\times T}},\overline{\mu|_{T\times T}}\right)
    =(\overline{\alpha_1},\overline{\mu_1},\overline{\alpha_2},\overline{\mu_2}).
    \end{align*}
    By Theorem \ref{thm:Kb_directproduct}, 
    \[
    M_b(N\times T) \simeq M_b(N)\times (T^{\ab}\otimes N^{\ab})\times M_b(T).
    \]
    Since $T_+^{\ab}\otimes N_+^{\ab} = 0$ and  
    $T^{\ab}$ and $N^{\ab}$ are  quotients of $T_+^{\ab}$ and $N_+^{\ab}$, respectively, it follows that 
    $T^{\ab}\otimes N^{\ab}=0$.
\end{proof}

The following result is an immediate consequence Theorem \ref{thm:Mb(NxT)}
and the results of this section. 

\begin{cor}
    \label{Mbcyclicnonbicyclic}
    Let $n\geq1$ odd, $d$ a divisor of $n$ and $m\geq1$, then 
    \[
    M_b(C_{(n,d)}\times C_{(2^{m+1},2)})\simeq \Z/(2)\times\Z/(2)\times \Z/(d).
    \]
\end{cor}


\subsection{Direct product of some bicyclic skew braces}

\begin{thm}
\label{thm:directproductbicyclic}
Let $N=C_{(m,d_1)}$ and $T=C_{(k,d_2)}$ be bicyclic skew braces. Let 
$g=\gcd(m,k)$. Assume that $g=\gcd(d_1,d_2)$. 
Then 
\[ 
M_{b}(N\times T)\simeq \Z/(d_1)\times \Z/(d_2)\times\Z/(g)^3
\]
\end{thm}


\begin{proof}
It is enough to show that the sequence $S(N\times T,\iota,\delta)$ is exact and splits. 
By Theorem \ref{thm:directproduct},
\[
M((N\times T)_+)\times M((N\times T)_\circ)\simeq\Z/(g)^2.
\]
Let $\xi$ be a $g$-th primitive root of unity. Let 
\begin{gather*}
\phi\colon \Z/(g)^2\to M_{b}(N\times T),
\quad
(l,u)\mapsto \overline{(\alpha(l),\mu(u))},
\shortintertext{where}
\alpha(l)((n_1,t_1),(n_2,t_2))=\xi^{lt_1n_2},
\quad 
\mu(u)((n_1,t_1),(n_2,t_2))=\xi^{ut_1n_2}.
\end{gather*}
Routine calculations show that $(\alpha(l),\mu(u))$ are factor sets of skew braces and
that $\delta\phi=\id$.
\end{proof}
\begin{exa}
    For distinct odd prime numbers $p_1,\dots,p_k,q_1,\dots , q_l$, 
    let 
    \[
    d=p_1\dots p_kq_1\dots q_l.
    \]
    If  
    $m_1,\dots, m_k,r_1,\dots,r_l\geq 1$, then
    \[
    M_b\left(C_{(p_1^{m_1}\dots p_k^{m_k}q_1\dots q_l,d)}\times C_{(p_1\dots p_kq_1^{r_1}\dots q_l^{r_l},d)}\right)\simeq \left(\Z/(d)\right)^5.
    \]
\end{exa}

\subsection{Skew braces of order $p^2$}

\begin{pro}[Proposition 2.4 of \cite{MR3320237}]
\label{pro:Bachiller}
    Let $p$ be a prime number. Up to isomorphism, skew braces of abelian type 
    of order $p^2$ are either trivial, cyclic with abelian multiplicative group, or have additive group isomorphic to $\Z/(p)\times \Z/(p)$ and multiplicative group defined as 
    \[(x_1,x_2)\circ(y_1,y_2)=(x_1+y_1+x_2y_2, x_2+y_2).
    \]
    In the last case, the multiplicative group is isomorphic to $\Z/(p)\times\Z/(p)$ if $p$ is odd and $\Z/(4)$ otherwise.
\end{pro}

\begin{notation}
For a prime number $p$, let  
$\B_p$ be the skew brace over $\Z/(p)\times\Z/(p)$ of Proposition \ref{pro:Bachiller}. 
\end{notation}


\begin{lem}
\label{lem:unlabeled}
Let $p$ be an odd prime number. The sequence $\sigma_+(\B_p)$ is exact and splits.
\end{lem}

\begin{proof}
By Theorem \ref{thm:directproduct}, $M((\B_p)_+)=\langle\overline{\alpha}\rangle\simeq\Z/(p)$ and $M((\B_p)_\circ)=\langle\overline{\mu}\rangle\simeq\Z/(p)$, where
\begin{align*}
&\alpha\colon (\B_p)_+\times (\B_p)_+\to\C^{\times},
\quad
((x_1,x_2),(y_1,y_2))\mapsto \xi^{x_2 y_1},\\
&\mu\colon (\B_p)_+\times (\B_p)_+\to\C^{\times},
\quad
((x_1,x_2),(y_1,y_2))\mapsto \xi^{x_2\left(y_1-\frac{y_2(y_2-1)}{2}\right)},
\end{align*}
for some $p$-th root of unity $\xi$. The formula for $\mu$ comes from the fact that for all $x_1,x_2\in \Z/(p)$, $(x_1,x_2)=(1,0)^{\circ \left(x_1-\frac{x_2(x_2-1)}{2}\right)}\circ (0,1)^{\circ x_2}$. Straightforward computations show that
\begin{align*}
        \alpha(y,z)\alpha(x\circ y,\lambda_x(z))^{-1}\alpha(x,\lambda_x(z))&=\xi^{-y_2x_2z_2}\\
        &=\mu(x,y)\mu(x,y+z)^{-1}\mu(x,z).
    \end{align*}
    
Hence the pair $(\alpha,\mu)$ is a factor set of skew braces. Since $\alpha^p=\mu^p=1$, 
$\overline{(\alpha,\mu)}$ has order $p$. 
Thus 
$s\colon M((\B_p)_+)\to M_{b}(\B_p)$, $\overline{\alpha}\mapsto\overline{(\alpha,\mu)}$, 
is a section of $\delta$. 
\end{proof}

\begin{lem}
\label{lem:FSBpdescription}
Let $p$ be an odd prime number and 
$\xi_1,\xi_2,\xi_3$ be $p$-th root of unity. Define the function $f:\B_p\times \B_p\to \C^{\times}$ by
\begin{equation*}
    f(x,y)=\xi_1^{y_1x_2+y_2\frac{x_2(x_2-1)}{2}}
    \xi_2^{y_2\left(x_1-\frac{x_2(x_2-1)}{2}\right)} \xi_3^{x_2y_2}.
\end{equation*}
for all $x,y\in \B_p$. The element $(1,f)$ is a factor set of the datum $(\B_p,\C^{\times})$.
\end{lem}

\begin{proof}
Let $x,y,z\in \B_p$
Clearly $f(0,y)=f(x,0)=1$. It is straight forward that $f(x,y+z)=f(x,y)f(x,z)$.
Then straightforward computations show that
\begin{align*}
    f(x\circ y,z)&=f(x,z)f(y,z)\xi_1^{z_2x_2y_2},\\
    f(x,y\circ z)&=f(x,y)f(x,z)f(x,(y_2z_2,0))=f(x,y)f(x,z)\xi_1^{z_2x_2y_2}.
\end{align*}
Therefore $f(y,z)f(x\circ y,z)^{-1}f(x,y\circ z)f(x,y)^{-1}=1$.
\end{proof}

\begin{pro}
    Let $p$ be an odd prime number. Then $S_+(\B_p)\simeq \Z/(p)\times\Z/(p)$.
\end{pro}

\begin{proof}
Let $c\in S_+(\B_p)$, $x,y\in \B_p$, then by Proposition \ref{pro:bilinearfactorset} there exists a representative $(1,f)$ of $c$ such that $Im(f)\subseteq G_p$ and $f(x,y)= f(x,e_1)^{y_1}f(x,e_2)^{y_2}$.
Combining equation \eqref{eq:leftlin} of Proposition \ref{pro:bilinearfactorset} and Equation \eqref{eq:FS4} one obtains
\begin{equation}
\label{eq:fe1e1trivial}
1=f(e_2,e_2)f(e_1\circ e_2,e_2)^{-1}f(e_1,e_2\circ e_2)f(e_1,e_2)^{-1}=f(e_1,e_1).
\end{equation}
Notice that $x=e_1^{\circ x_1}\circ e_2^{\circ x_2}\circ(-\frac{x_2(x_2-1)}{2},0)$. Thus, by equation \eqref{eq:leftlin} of Proposition \ref{pro:bilinearfactorset} and \eqref{eq:fe1e1trivial},
\[
f(x,e_1)= f(e_2,e_1)^{x_2}
\]
Again, by \eqref{eq:leftlin},
\[f(x,e_2)=f(e_1,e_2)^{x_1}f(e_1,e_2)^{\frac{x_2(x_2-1)}{2}}f(e_2^{\circ x_2},e_2).
\]
Moreover, by \eqref{eq:FS4}
\[f(e_2^{\circ x_2},e_2)= f(e_2,e_2)f(e_2^{\circ(x_2-1)},e_2)f(e_2,e_1)^{x_2-1}.
\]
Therefore, 
\[f(e_2^{\circ x_2},e_2)=f(e_2,e_2)^{x_2}f(e_2,e_1)^{\frac{x_2(x_2-1)}{2}}.
\]
Hence,
\begin{equation*}
    f(x,y)=f(e_2,e_1)^{y_1x_2+y_2\frac{x_2(x_2-1)}{2}}\\
    f(e_1,e_2)^{y_2(x_1-\frac{x_2(x_2-1)}{2})} f(e_2,e_2)^{x_2y_2}.
\end{equation*}
This and Lemma \ref{lem:FSBpdescription} implies that any element of $S_+(\B_p)$ has a representative of the form $(1,f)$ with
\begin{equation*}
    f(x,y)=\xi_1^{y_1x_2+y_2\frac{x_2(x_2-1)}{2}}
    \xi_2^{y_2\left(x_1-\frac{x_2(x_2-1)}{2}\right)} \xi_3^{x_2y_2}.
\end{equation*}
for all $x,y\in \B_p$ with $\xi_1,\xi_2,\xi_3$ some $p$-th root of unity. Let $(1,f)$ and $(1,g)$ be two such representatives of the same element of $S_+(\B_p)$. Denote respectively $\xi_1,\xi_2,\xi_3$ and $\theta_1,\theta_2,\theta_3$ their associated $p$-th root of unity. Then there is a group homomorphism $h\colon(\B_p)_+\mapsto\C^{\times}$ such that $g=\partial_{\circ}h f$. Notice that $Im(h)\subset G_p$ and $\partial_{\circ}h(x,y)=h(e_1)^{x_2y_2}$, thus $\partial_{\circ}h$ is uniquely determined by some $p$-th root of unity $\xi$ such that $\partial_{\circ}h(x,y)=\xi^{x_2y_2}$. Therefore, $g=\partial_{\circ}h f$ if and only if $\xi_3\xi=\theta_3$, $\xi_1=\theta_1$ and $\xi_2=\theta_2$. Thus the elements of $S_+(\B_p)$ are uniquely determined by two $p$-th roots of unity.
\end{proof}

\begin{cor}
    Let $p$ be an odd prime number. Then $M_b(\B_p)\simeq \Z/(p)^3$
\end{cor}

\section{Schur covers}
\label{coverings}

The results and proofs of the first two parts of this section (until Subsection \ref{subsection:examples}) are modeled on \cite[\S2.1]{MR1200015}. Good references for 
the theory of Schur covers are for example \cite{MR0681287} and \cite{MR1357169}. 


\begin{defn}
    Let $Q$ be a skew brace. A \emph{Schur cover} of $Q$ is a skew brace 
    $E$ such that
    there exists an annihilator extension
    \[
    \begin{tikzcd}
	0 & K & {E} & Q & 0
	\arrow[from=1-1, to=1-2]
	\arrow[from=1-2, to=1-3]
	\arrow[from=1-3, to=1-4]
	\arrow[from=1-4, to=1-5]
    \end{tikzcd}
    \]    
    with $K\subseteq E'$ and $K\simeq M_b(Q)$. 
\end{defn}

\begin{lem}
\label{lem:restricts}
    Let $A$ be a finite skew brace and $K\subseteq A$ a subskewbrace. Then every
    skew brace homomorphism $A\to\C^{\times}$ restricts to the trivial homomorphism 
    on $K$ if and only if
    $K\subseteq A'$.
\end{lem}

\begin{proof}
    We only prove the non-trivial implication. 
    Assume that there is an element $a\in K\cap(A\setminus A')$. 
    Let $k$ be the order of $\overline{a}$, where $\overline{a}$ denotes the equivalence class of $a$ in the finite 
    abelian group $A^{\ab}$. 
    Let $\xi$ be a $k$-th primitive root of unity and $\alpha\colon \langle\overline{a}\rangle\to\C^{\times}$, $\overline{a}\mapsto\xi$. 
    Since $\C^{\times}$ is injective, the homomorphism $\alpha$ can be extended to a group homomorphism
    $A^{\ab}\to\C^{\times}$. The composition with the natural map yields a 
    skew brace homomorphism $A\to\C^{\times}$ such that $a\mapsto\xi$. 
\end{proof}

\begin{pro}
\label{pro:coversTra}
Let
\[
    \begin{tikzcd}
	0 & K & {E} & Q & 0
	\arrow[from=1-1, to=1-2]
	\arrow[from=1-2, to=1-3]
	\arrow[from=1-3, to=1-4]
	\arrow[from=1-4, to=1-5]
    \end{tikzcd}
    \]
    be an annihilator extension of finite skew braces. Then $E$ is a cover of $Q$
    if and only if the transgression map $\Tra\colon \Hom(K,\C^{\times})\to M_b(Q)$ is a group isomorphism.
\end{pro}

\begin{proof}
    Suppose that $E$ is a cover. Since $K\subseteq E'$, every skew brace homomorphism $E\to\C^{\times}$ restricts to the trivial skew brace homomorphism $K\to\C^{\times}$.
By Theorem \ref{thm:HoschildSerre}, the transgression is injective. In addition, since $|K|=|M_b(Q)|$, $\Tra$ must be an isomorphism. 

Conversely, Assume that $\Tra$ is an isomorphism. By Theorem \ref{thm:HoschildSerre}, the map $\Res \colon \Hom(E,\C^{\times})\to \Hom(K,\C^{\times})$ is trivial. Lemma  \ref{lem:restricts} implies that $K\subseteq E'$. Moreover, $K\simeq \Hom(K,\C^{\times})\simeq M_b(Q)$.
\end{proof}
The following theorem and proof are modeled on Theorem 2.1.4 of \cite{{MR1200015}}.
\begin{thm}
\label{thm:existence}
    Any finite skew brace admits at least one Schur cover. 
\end{thm}

\begin{proof}
    Let $Q$ be a finite skew brace. Since $M_b(Q)$ is a finite abelian group by Theorem \ref{thm:M(Q)finite}, 
    we write 
    \[
    M_b(Q)=\langle c_1\rangle\times\cdots\times\langle c_m\rangle
    \]
    for some $c_1,\dots,c_m$. For each $i\in\{1,\dots,m\}$ let $d_i$ be the order of $c_i$. By Theorem \ref{thm:M(Q)finite}, Each
    $c_i$ has a representative $(\alpha_i,\mu_i)$ with 
    $\alpha_i,\mu_i\colon Q\times Q\to\C^{\times}$ such that 
    $\Img(\alpha_i)\subseteq G_{d_i}$ and $\Img(\mu_i)\subseteq G_{d_i}$. 
    Let $G=G_{d_1}\times\cdots\times G_{d_m}$ and 
    \[
    \iota_i\colon H^2_{b}\left(Q,G_{d_i}\right)\to G
    \]
    where G is identified with $M_b(Q)$ and $\iota_i$ is the homomorphism induced by the inclusion $G_{d_i}\subseteq\C^{\times}$. 
    
    For
    each $i$, let $\beta_i\in H^2_b(Q,G_{d_i})$ be 
    the equivalence class of $(\alpha_i,\mu_i)$ considered as an element of $Z^2_b(Q,G_{d_i})$. Then 
    $c_i=\iota_i(\beta_i)$, 
    
    Let $\beta\in H^2_{b}(Q,G)$ be the image of $\left(\beta_1,\dots , \beta_m\right)$ under the isomorphism 
    \[
    \prod_{i=1}^m H^2_{b}\left(Q,G_{d_i}\right)\to H^2_{b}\left(Q,G\right).
    \]
    Let 
    \[
    \begin{tikzcd}
    0\arrow[r] & G\arrow{r} & E\arrow[r]& Q\arrow[r]&0
    \end{tikzcd}
    \]
    be an annihilator extension associated with $\beta$. Let us show that $E$ is a covering skew brace of $Q$. 
    Consider the transgression homomorphism associated to $\beta$,
    \[
    \Tra\colon \Hom\left(G,\C^{\times}\right) \to M_{b}\left(Q\right).
    \]
    Let $p_i\colon G\to\C^{\times}$, $p(z_1,\dots, z_m)=z_i$. Then $\Tra\left(p_i\right) = c_i$ and hence $\Tra$ is surjective. 
    Since 
    \[
    \Hom\left(G,\C^{\times}\right)\simeq G\simeq M_{b}\left(Q\right)
    \]
    and $M_b(Q)$ is finite, 
    $\Tra$ is an isomorphism. 
    This concludes the proof by Proposition~\ref{pro:coversTra}.
\end{proof}

\begin{exa}
\label{ex:Schurcovercyclicprime}
In general, Schur covers are not unique up to isomorphism. For example, 
let $p$ be a prime number. Then the trivial skew brace $\Z/(p)$ 
has exactly two Schur covers, i.e., $C_{(p^2,p)}$ and $\B_p$. 
\end{exa}



    


\subsection{The relation between $M_b(Q)$ and $H_b^2(Q,A)$}

We now aim to obtain an upper bound for the number of Schur covers of a given skew brace. 

\begin{lem}
\label{lem:imageInflation}
    Let $\pi\colon Q_1 \to Q_2$ be a surjective homomorphism of skew braces and 
    \[
    \begin{tikzcd}[column sep=1.7em]
	0 & K & {B_1} & {Q_1} & 0 & 0 & K & {B_2} & {Q_2} & 0
	\arrow[from=1-1, to=1-2]
	\arrow[from=1-2, to=1-3]
	\arrow["{f_1}", from=1-3, to=1-4]
	\arrow[from=1-4, to=1-5]
	\arrow[from=1-6, to=1-7]
	\arrow[from=1-7, to=1-8]
	\arrow["{f_2}", from=1-8, to=1-9]
	\arrow[from=1-9, to=1-10]
\end{tikzcd}\]
be annihilator extensions that correspond to 
$c_1\in H_b^2(Q_1,K)$ and $c_2\in H_b^2(Q_2,K)$, respectively. 
Then $c_1$ is the image of $c_2$ under the inflation map
\[
H^2(Q_2,K)\to H^2(Q_1,K)
\] 
if and only if there exists a skew brace homomorphism 
$\phi\colon B_1\to B_2$, which renders the following diagram commutative
\[
\begin{tikzcd}
	0 & K & {B_1} & {Q_1} & 0 \\
	0 & K & {B_2} & {Q_2} & 0
	\arrow[from=1-1, to=1-2]
	\arrow[from=1-2, to=1-3]
	\arrow["{f_1}", from=1-3, to=1-4]
	\arrow[from=1-4, to=1-5]
	\arrow[from=2-1, to=2-2]
	\arrow[from=2-2, to=2-3]
	\arrow["{f_2}", from=2-3, to=2-4]
	\arrow[from=2-4, to=2-5]
	\arrow["\phi", from=1-3, to=2-3]
	\arrow["\pi", from=1-4, to=2-4]
	\arrow[equal,from=1-2, to=2-2]
\end{tikzcd}
\]
\end{lem}

\begin{proof}
Let $t_1$ and $t_2$ be respective transversals of $f_1$ and $f_2$. 
For $i\in\{1,2\}$ consider the factor sets
\[
\alpha_i(x,y)=t_i(x)+t_i(y)-t_i(x+y)\quad \text{and} \quad \mu_i(x,y)=t_i(x)\circ t_i(y)\circ t_i(x\circ y)'.
\]
For the ``if" part, one sees that by the commutativity of the diagram,  
\[
f_2(\phi(t_1(x)))=\pi(x)=f_2(t_2(\pi(x)))
\]
so that there exists a map $h\colon Q_1\to K$ with $h(0)=0$ such that 
\[
\phi(t_1(x))=h(x)+ t_2(\pi(x))=h(x)\circ t_2(\pi(x)).
\]
Therefore
\begin{align*}
    \alpha_1(x,y)&=\phi(\alpha_1(x,y))
    =\partial_+h(x,y)\alpha_2(\pi(x),\pi(y)).
\end{align*}
Similarly one has 
\[
\mu_1(x,y)=\partial_{\circ}h(x,y)\mu_2(\pi(x),\pi(y)).
\]
Conversely, assume that 
\begin{align*}
&\alpha_1(x,y)=\partial_+h(x,y)\alpha_2(\pi(x),\pi(y)),\\
&\mu_1(x,y)=\partial_\circ h(x,y) \mu_2(\pi(x),\pi(y))
\end{align*}
for some map $h\colon Q\to K$ with $h(0)=0$. Straightforward computations show that the map 
\[
\phi\colon B_1\to B_2,\quad k+t_1(x)\mapsto k+h(x)+t_2(\pi(x))
\]
has the required properties.
\end{proof}

If the two annihilator extensions 
\[
\begin{tikzcd}[column sep=1.7em]
0\arrow[r] & K \arrow[r] & B_1\arrow[r] & Q\arrow[r] & 0
\end{tikzcd}
\quad \text{and}\quad\begin{tikzcd}[column sep=1.7em]
0\arrow[r] & K \arrow[r] & B_2\arrow[r] & Q\arrow[r] & 0
\end{tikzcd}
\]
are equivalent, then $B_1'\cap K=0$ if and only if $B_2'\cap K=0$. 
We denote by $I_b^2(Q,K)$ the subset of $H^2_b(Q,K)$ 
that consists of the equivalence classes of 
2-cocycles determined by an annihilator extension  
$\begin{tikzcd}[column sep=1.7em]
0\arrow[r] & K \arrow[r] & B\arrow[r] & Q\arrow[r] & 0
\end{tikzcd}$ 
such that $B'\cap K=0$. 

\begin{lem}
\label{lem:inflationinjective}
    The restriction of the inflation map 
    \[
    \Ext(Q/Q',K)\to H^2_b(Q,K),
    \]
    where $\Ext(Q/Q',K)$ denotes the subset of $H^2(Q/Q',K)$ that 
    yields abelian group extensions,
    is an injective homomorphism whose image is $I_b^2(Q,K)$.
\end{lem}

\begin{proof}
    Denote by $\tau\colon\Ext(Q/Q',K)\to H^2_b(Q,K)$ 
    the inflation map. 
    
    For $c_2\in\Ext(Q/Q',K)$ let $c_1=\tau(c_2)$. 
    By Lemma \ref{lem:imageInflation} applied to the corresponding 
    annihilator extensions, there exists a commutative diagram 
    \[
    \begin{tikzcd}
	0 & K & {B_1} & {Q} & 0 \\
	0 & K & {B_2} & {Q/Q'} & 0
	\arrow[from=1-1, to=1-2]
	\arrow[from=1-2, to=1-3]
	\arrow["{f_1}", from=1-3, to=1-4]
	\arrow[from=1-4, to=1-5]
	\arrow[from=2-1, to=2-2]
	\arrow[from=2-2, to=2-3]
	\arrow["{f_2}", from=2-3, to=2-4]
	\arrow[from=2-4, to=2-5]
	\arrow["\phi", from=1-3, to=2-3]
	\arrow["\pi", from=1-4, to=2-4]
	\arrow[equal, from=1-2, to=2-2]
    \end{tikzcd}
\]
where $\pi\colon Q\to Q/Q'$ is the canonical map. 
Assume first that $c_2\in\Ker(\tau)$. Then there is a splitting $t_1\colon Q\to B_1$. Since $B_2$ is an abelian group, the homomorphism $\phi t_1$ factors through  
$t_2\colon Q/Q'\to B_2$. By the commutativity of the diagram, $f_2t_2$ is the identity. Thus $\tau$ is injective.

Back to the general case, notice that if $k\in K\cap B_1'$ then $k=\phi(k)=0$. Therefore $c_1\in I_b^2(Q,K)$.

Let $c\in I_b^2(Q,K)$ and consider the associated 
annihilator extension 
\[
\begin{tikzcd}
0\arrow[r] & K \arrow[r] & B\arrow[r,"f"] & Q\arrow[r] & 0
\end{tikzcd}
\]
corresponding to $c$. 
Since $B'\cap K=0$, we identify $K$ with $(B'+K)/B'$ in $B/B'$. With this identification, 
we obtain the commutative diagram
\[\begin{tikzcd}
	0 & K & {B} & {Q} & 0 \\
	0 & K & {B/B'} & {Q/Q'} & 0
	\arrow[from=1-1, to=1-2]
	\arrow[from=1-2, to=1-3]
	\arrow["{f}", from=1-3, to=1-4]
	\arrow[from=1-4, to=1-5]
	\arrow[from=2-1, to=2-2]
	\arrow[from=2-2, to=2-3]
	\arrow["{f'}", from=2-3, to=2-4]
	\arrow[from=2-4, to=2-5]
	\arrow[from=1-3, to=2-3]
	\arrow[ from=1-4, to=2-4]
	\arrow[equal, from=1-2, to=2-2]
\end{tikzcd}
\]
with $f'(x+B')=f(x)+Q'$. Lemma \ref{lem:imageInflation} concludes the proof
\end{proof}

The following lemma is modeled on \cite[Lemma 2.1.18]{MR1200015}.

\begin{lem}
\label{lem:split_b}
    Let $Q$ be a finite skew brace. 
    The exact sequence 
    \[ 
    \begin{tikzcd}
    0\arrow[r] & B_b^2(Q,\C^{\times}) \arrow[r] & Z_b^2(Q,\C^{\times})\arrow[r] & H_b^2(Q,\C^{\times})\arrow[r] & 0
    \end{tikzcd}
    \]
    splits.
\end{lem}

\begin{proof}
Let $E$ be a Schur cover of $Q$ with factor set $(\alpha,\mu)$; 
see Theorem \ref{thm:existence}. By Proposition \ref{pro:coversTra}, 
the associated transgression map 
$\Hom(M_b(Q),\C^{\times})\to M_b(Q)$ 
is a group isomorphism. For every  
$d\in M_b(Q)$ there is a map $\phi_d\in \Hom(M_b(Q),\C^{\times})$ such that $(\phi_d \alpha,\phi_d \mu)\in Z^2_b(Q,\C^{\times})$ and $\overline{(\phi_d \alpha,\phi_d \mu)}=d$. Straightforward computations show that the map $d\mapsto (\phi_d \alpha,\phi_d \mu)$ is a splitting.
\end{proof}

\begin{notation}
   Let $K$ be a finite abelian group and $\widehat{K}=\Hom(K,\C^{\times})$. For all $c\in H_b^2(Q,K)$, write $\psi(c)\colon \widehat{K}\to M_b(Q)$ to denote the transgression map associated to $c$.
   Write $\phi\colon\Ext(Q/Q',K)\to H_b^2(Q,K)$ 
   to denote 
   the restriction of the inflation map.
\end{notation}

\begin{lem}
\label{lem: surjRestr}
    Let $K$ be an abelian group that is a sub skew brace of a finite skew brace $Q$. If the restriction map $\Hom(Q,\C^{\times})\to \Hom(K,\C^{\times})$ is surjective if and only if $Q'\cap K=0$.
\end{lem}

\begin{proof}
    Suppose by contrapositive that there exists $x\in Q'\cap K$ such that $x\neq 0$. 
    Then there is a group homomorphism $\phi\colon K\to \C^{\times}$ such that $\phi(x)\neq 1$. 
    This homomorphism cannot be a restriction of an element of $\Hom(Q,\C^{\times})$, as every such element restrict to $1$ on $Q'$.
    Conversly, Assume $B'\cap K=0$, the claim follows by identifying $K$ with $(B'+K)/B'$ in $B/B'$ and using the injectivity of $\C^{\times}$.
\end{proof}

The following result is \cite[Theorem 2.1.19]{MR1200015} in the context of skew braces. 

\begin{thm}
Let $K$ be a finite abelian group. Then
\begin{equation}
\label{eq:extsplit}
\begin{tikzcd}[column sep=1.7em]
0\arrow[r] & \Ext(Q/Q',K) \arrow[r,"\phi"] & H_b^2(Q,K)\arrow[r,"\psi"] & \Hom(\widehat{K}, M_b(Q))\arrow[r] & 0
\end{tikzcd}
\end{equation}
is exact and splits. 
\end{thm}

\begin{proof}
First we show that \eqref{eq:extsplit} 
is exact at $H_b^2(Q,K)$. Let $c\in H_b^2(Q,K)$ and consider an associated 
annihilator extension 
$\begin{tikzcd}[column sep=1.7em]
0\arrow[r] & K \arrow[r] & B\arrow[r] & Q\arrow[r] & 0
\end{tikzcd}$. 
By the Hoschild--Serre exact sequence (Theorem \ref{thm:HoschildSerre}), 
$\psi(c)=0$ if and only if the restriction map 
\[
\Hom(B,\C^{\times})\to\Hom(K,\C^{\times})
\]
is surjective. The latter is equivalent to $K\cap B'=0$. 
The injectivity of $\phi$ comes from Lemma \ref{lem:inflationinjective}. It is left to construct a splitting 
\[
\theta\colon \Hom(\widehat{K}, M_b(Q))\to H_b^2(Q,K).
\]
By Lemma \ref{lem:split_b}, 
\[
Z_b^2(Q,\C^{\times})= B_b^2(Q,\C^{\times})\oplus X
\]
for some subgroup $X$ of $Z_b^2(Q,\C^{\times})$ 
isomorphic to $M_b(Q)$. This implies that
for every $\gamma\in \Hom(\widehat{K},M_b(Q))$ 
and $\xi\in\widehat{K}$, 
$\gamma(\xi)\in M_b(Q)\simeq X$. 
By abuse of notation, we will identify 
$\gamma(\xi)$ with its image in  
$Z_b^2(Q,\C^{\times})$ 
for $\gamma\in\Hom(\widehat{K},M_b(Q))$
and $\xi\in\widehat{K}$. 

For $x,y\in Q$, it is straightforward to see that the map 
\[
\chi_{\gamma}(x,y)\colon \widehat K\to \C^{\times}\times\C^{\times}, \quad \xi\mapsto \gamma(\xi)(x,y),
\]
is a group homomorphism. To any $\gamma\in \Hom(\widehat{K},M_b(Q))$ 
we can associate the maps
\begin{align*}
&\alpha_{\gamma}: Q\times Q\to \Hom(\widehat{K},\C^{\times}),\quad (x,y)\mapsto \pi_1\chi_{\gamma}(x,y),\\
&\mu_{\gamma}\colon Q\times Q\to \Hom(\widehat{K},\C^{\times}),\quad (x,y)\mapsto \pi_2\chi_{\gamma}(x,y),
\end{align*}
where $\pi_1\colon \C^{\times}\times\C^{\times}\to \C^{\times}$ and $\pi_2\colon \C^{\times}\times\C^{\times}\to \C^{\times}$ denote respectively 
the projection onto the first and second coordinates. Straightforward computations show that $(\alpha_{\gamma},\mu_{\gamma})\in Z_b^2(Q,\Hom(\widehat{K},\C^{\times}))$. 
Since $K$ is finite, the map 
\[
\tau\colon K\to\Hom(\widehat{K},\C^{\times}), \quad a\mapsto (\xi\mapsto \xi(a)),
\]
is an isomorphism. 
We define
$\theta(\gamma)=\overline{(\tau^{-1}\alpha_{\gamma},\tau^{-1}\mu_{\gamma})}$. It is straightforward to see that $\theta$ is an homomorphism. Then
\[
(\psi\theta(\gamma))(\xi)
=\overline{(\xi\tau^{-1}\alpha_{\gamma},\xi\tau^{-1}\mu_{\gamma})}
\]
for all $\xi\in\widehat{K}$. 
In addition, let $x,y\in Q$. Then 
\begin{align*}
    (\xi\tau^{-1}\alpha_{\gamma}(x,y),\xi\tau^{-1}\mu_{\gamma}(x,y))
    &=(\pi_1\chi_{\gamma}(x,y)(\xi),\pi_2\chi_{\gamma}(x,y))(\xi))\\
    &=(\pi_1\gamma(\xi)(x,y),\pi_2\gamma(\xi)(x,y))\\
    &= \gamma(\xi)(x,y).
\end{align*}
Thus $\psi\theta =\id$.
\end{proof}

The following corollary is modeled on \cite[Corollary 2.1.20]{MR1200015}. 

\begin{cor}
    Let $Q$ be a finite skew brace and $K$ be a finite abelian group. Then $H_b^2(Q,K)\simeq \left((Q/Q')\otimes K\right)\times (M_b(Q)\otimes K)$.
\end{cor}

\begin{proof}
This is a consequence of the splitting exact sequence $\eqref{eq:extsplit}$ and the fact that 
\[
\Hom(A,B)\simeq \Hom(B,A)\simeq A\otimes B\simeq \Ext(A,B)
\]
for all finite abelian groups $A$ and $B$.
\end{proof}

\begin{lem}[Lemma 2.5.13 of \cite{MR1200015}]
\label{lem: automorphism}
    Let $K$ be a finite abelian group. For any automorphism $\psi$ of $\Hom(K,\C^{\times})$, 
    there exists a group automorphism $\theta$ of $K$ such that for all $\xi\in \Hom(K,\C^{\times})$ and $k\in K$, 
    $\psi(\xi)(k)=\xi(\theta(k))$.
\end{lem}

\begin{thm}
\label{thm:number}
    Let $Q$ be a finite skew brace. If
    \[ Q/Q'\simeq \Z/(n_1)\oplus \dots \oplus \Z/(n_k) \quad \text{and} \quad M_b(Q)\simeq \Z/(m_1)\oplus \dots \oplus \Z/(m_l)
    \]
    then the number of non isomorphic skew brace 
    Schur covers of $Q$ is at most
    \[\prod_{i,j} \gcd(n_i,m_j).
    \]
\end{thm}

\begin{proof}
    Let $K=M_b(Q)$. For $c\in H_b^2(Q,K)$, let 
    \[
    \tau_c\colon\Hom(K,\C^{\times})\to M_b(Q)
    \]
    be the associated transgression map. 
    Let 
    \[
    E=K\times_{(\alpha,\mu)}Q,
    \quad
    F=K\times_{(\eta,\gamma)}Q
    \]
    be two Schur covers of $Q$. We write $c=\overline{(\alpha,\mu)}$ and $d=\overline{(\eta,\gamma)}$. 
    
    We may assume that $\tau_c=\tau_d$. In fact, by Proposition \ref{pro:coversTra}, the maps $\tau_c$ and $\tau_d$ are isomorphisms. Therefore to any $\xi\in \Hom(K,\C^{\times})$ we can associate a unique element $\xi^*\in \Hom(K,\C^{\times})$ such that $\tau_c(\xi)=\tau_d(\xi^*)$. The map $\xi\mapsto \xi^*$ is an automorphism of $\Hom(K,\C^{\times})$. By Lemma \ref{lem: automorphism}, there is an automorphism $\theta$ of $K$ such that 
    \[
    \xi^*(k)=\xi(\theta(k))\quad \text{for all $k\in K$}.
    \]
    Let $\kappa=\overline{(\theta\eta,\theta\gamma)}$. Then $\tau_{\kappa}=\tau_{d}$. 
    In addition, $K\times_{(\theta\eta,\theta\gamma)}Q$ is isomorphic to $E$ via $(k,x)\mapsto (\theta^{-1}(k),x)$. 
    This proves the claim. 
    
    Now  it is enough to show that $cd^{-1}\in I_b^2(Q,K)$, 
    since 
    \[
        |I_b^2(Q,K)|=|\Ext(Q/Q',K)|=\prod_{i,j} \gcd(n_i,m_j),
    \] 
    It is straightforward to see that the map 
    \[
    H(Q,K)\to \Hom(\Hom(K,\C^{\times}),M_b(Q)),
    \quad 
    f\mapsto \tau_f, 
    \]
    is a group homomorphism. Thus $\tau_{cd^{-1}}$ is trivial. Let $G$ be a skew brace 
    extension associated 
    with $cd^{-1}$. 
    By the Hoschild--Serre exact sequence (Theorem \ref{thm:HoschildSerre}), 
    the restriction $\Hom(G,\C^{\times})\to \Hom(K,\C^{\times})$ is surjective.  
    Then the theorem follows from Lemma \ref{lem: surjRestr}.
\end{proof}

\subsection{Some examples}
\label{subsection:examples}

\begin{pro}
\label{pro:perfectgroup}
    Let $G$ be a perfect group. If $H$ is a Schur cover group of $G$, 
    then $H$ is a skew brace Schur cover of $G$.
\end{pro}

\begin{proof}
    It follows from the fact that $M_b(G)\simeq M(G)$.
\end{proof}

\begin{pro}
\label{pro:coverbicyclic}
    Let $n\geq1$ be an integer 
    and $d$ a divisor of $n$ such that $C_{(n,d)}$ 
    is a bicyclic brace. Then $C_{(nd,d)}$ is a skew brace Schur cover of $C_{(n,d)}$. 
\end{pro}

\begin{proof}
    By Corolary \ref{cor:Schur_bicyclic}, $M_b(C_{(n,d)})\simeq \Z/(d)$. 
    A straightforward calculation shows that the sequence 
    \[
    \begin{tikzcd}
    	0 & {\Z/(d)} & {C_{(nd,d)}} & {C_{(n,d)}} & 0
    	\arrow[from=1-1, to=1-2]
    	\arrow["n", from=1-2, to=1-3]
    	\arrow["\pi", from=1-3, to=1-4]
    	\arrow[from=1-4, to=1-5]
    \end{tikzcd},
    \]
    where $\pi(\overline{x})=\overline{x}$, is exact. 
\end{proof}

Isoclinism of skew braces was introduced in \cite{zbMATH07779852}. 
We define two collections of maps $\phi_+$ and $\phi_*$ that associate respectively to every skew brace $B$ the maps
\begin{align}
    \label{eq:commutator}&\phi^B_{+}\colon (B/\Ann B)^2\to B',&&(\overline{a},\overline{b})\mapsto [a,b]_+,\\
    \label{eq:star}&\phi^B_{*}\colon (B/\Ann B)^2\to B',&&(\overline{a},\overline{b})\mapsto a*b. 
\end{align}

We say that the braces $A$ and $B$ are \emph{isoclinic} 
    if there are two isomorphisms 
    $\xi\colon A/\Ann A\to B/\Ann B$ and $\theta\colon A'\to B'$ 
    such that
    \begin{equation}
        \label{eq:isoclinism}
    \begin{tikzcd}
	{A'} & {(A/\Ann A)^2} & {A'} \\
	{B'} & {(B/\Ann B)^2} & {B'}
	\arrow["\phi_+^A"', from=1-2, to=1-1]
	\arrow["{\theta }", from=1-1, to=2-1]
	\arrow["{\phi_+^B}", from=2-2, to=2-1]
	\arrow["{\phi_*^A}", from=1-2, to=1-3]
	\arrow["{\phi_*^B}"', from=2-2, to=2-3]
	\arrow["\theta", from=1-3, to=2-3]
	\arrow["\xi\times\xi", from=1-2, to=2-2]
    \end{tikzcd}
    \end{equation}
    commutes. We call the pair $(\xi,\theta)$ a skew brace \emph{isoclinism}.
    
\begin{lem}
\label{AnnOfAnExtension}
Let $(\alpha,\mu)\in Z_b^2(Q,K)$ and $E=K\times_{(\alpha,\mu)}Q$. 
Then
\begin{align*}
    \Ann E=\{(k,q):q\in\Ann Q,\;\alpha(q,x)&=\alpha(x,q)\\&=\mu(q,x)=\mu(x,q)\;\forall x\in Q\}.
\end{align*}
\end{lem}

\begin{proof}
Let $(k,q)\in E$, $b\in K$ and $x\in Q$. A direct calculation shows that 
\[
(k,q)+(b,x)=(b,x)+(k,q)
\]
if and only if 
\[
    (kb\alpha(q,x),q+x)=(bk\alpha(x,q),x+q).
\]
In addition, $(k,q)\circ(b,x)=(b,x)\circ (k,q)$ if and only if
\[
(kb\mu(q,x),q\circ x)=(bk\mu(x,q),x\circ q).
\]
Finally,
$(k,q)+(b,x)=(k,q)\circ (b,x)$ if and only if 
\[
(kb\alpha(q,x),q+x)=(kq\mu(q,x),q\circ x).
\]
From this, the claim follows. 
\end{proof}

\begin{rem}
\label{TrivOfAnExtension}
Let 
$\begin{tikzcd}
0\arrow[r] & K \arrow[r] & E\arrow[r] & Q\arrow[r] & 0
\end{tikzcd}$ 
be an annihilator extension of braces. Let $(\alpha,\mu)$ 
be a factor set of the extension. A direct calculation shows that 
the generators of $E'$ are exactly the elements
\begin{gather*}
(\alpha(q,x)\alpha(q,-q)^{-1}\alpha(x,-x)^{-1}\alpha(q+x,-q)\alpha(q+x-q,-x),[q,x]_+),\\
(\alpha(q,-q)^{-1}\alpha(x,-x)^{-1}\mu(q,x)\alpha(-q,q\circ x)\alpha(-q+q\circ x,-x),q*x)
\end{gather*}
with $q,x\in Q$.
\end{rem}

\begin{thm}
\label{thm:isoclinism}
    Any two Schur covers of a finite skew brace are isoclinic. 
\end{thm}

\begin{proof}
Let $Q$ be a finite skew brace and $K=M_b(Q)$. 
Let 
\[
E=K\times_{(\alpha,\mu)}Q,
\quad
F=K\times_{(\eta,\gamma)}Q,
\]
be two Schur covers of $Q$.

Assume that $K\simeq \langle c_1\rangle\times\dots \times \langle c_m\rangle$,
where each $c_i$ is a complex $n_i$-th root of unity. 
There exist $(\alpha_i,\mu_i)\in Z^2_b(Q,\langle c_i\rangle)$ and $(\eta_i,\gamma_i)\in Z^2_b(Q,\langle c_i\rangle)$ for all $1\leq i\leq m$ such that
\begin{align*}
&\alpha(x,y)=(\alpha_1(x,y),\dots,\alpha_m(x,y)),
&&\mu(x,y)=(\mu_1(x,y),\dots,\mu_m(x,y)),\\
&\eta(x,y)=(\eta_1(x,y),\dots,\eta_m(x,y)),
&&\gamma(x,y)=(\gamma_1(x,y),\dots,\gamma_m(x,y)).
\end{align*}
Since each $c_i\in\C^{\times}$, 
the pair $(\alpha_i,\mu_i)$ is a representative of some $g_i\in K$. Similarly, 
each $(\eta_i,\gamma_i)$ is a representative of some $f_i\in K$. 
By Proposition \ref{pro:coversTra}, the transgression map 
\[
\begin{tikzcd}
\Hom\left(K,\C^{\times}\right)\arrow[r,"\Tra"] & K
\end{tikzcd}
\]
with respect to $E$ is an isomorphism. 
Similarly, the transgression map with respect to $F$ is an isomorphism. 
There are homomorphisms $\phi_i\colon K \to\C^{\times}$ such that 
$(\phi_i\alpha,\phi_i\mu)$ is a representative of $f_i$. Similarly, there are homomorphisms $\psi_i\colon K \to\C^{\times}$ such that 
$(\psi_i\eta,\psi_i\gamma)$ is a representative of $g_i$. 
Thus there are maps $h_i\colon Q\to\C^{\times}$ such that $\left(\phi_i\alpha\partial_+h_i, \phi_i\mu\partial_{\circ} h_i\right)=(\eta_i,\gamma_i)$. 

We claim that $(k,q)\in\Ann(E)$ if and only if $(k,q)\in\Ann(F)$. To prove this let $(k,q)\in\Ann(E)$. By Lemma \ref{AnnOfAnExtension}, 
$q\in\Ann(Q)$ and 
\[
\alpha_i(q,x)=\alpha_i(x,q)=
\mu_i(q,x)=\mu_i(x,q)
\]
for all $x\in Q$ and $i\in\{1,\dots,m\}$. Let $x\in Q$ and  $i\in\{1,\dots,m\}$. Then 
\[
\eta_i(q,x) = \phi_i\left(\alpha_i\left(q,x\right)\right)h_i\left(x\right)h_i\left(q+x\right)^{-1}h_i\left(q\right)=\eta_i(x,q).
\]
Similarly, one proves that 
$\gamma_i(q,x)=\gamma_i(x,q)$ and 
$\eta_i(q,x)=\gamma_i(q,x)$.

The following map is a well-defined surjective
skew brace homomorphism: 
\[
\xi\colon E/\Ann(E)\to F/\Ann(F).
\quad
(k,q)+\Ann(E)\mapsto (k,q)+\Ann(F).
\]

To show that $E'\simeq F'$, consider the skew brace 
$\tilde{F}=(\C^{\times})^m\times_{(\eta,\gamma)}Q$. Then, define the maps 
\begin{align*}
&\phi\colon K\to (\C^{\times})^m,&& k \mapsto (\phi_1(k),\dots , \phi_m(k)), \\
&h\colon Q\to (\C^{\times})^m,&& q \mapsto (h_1(q),\dots , h_m(q)),\\
&\theta \colon E\to \tilde{F},&& (k,q)\mapsto (\phi(k)h(q)^{-1},q).
\end{align*}
A straightforward calculation shows that 
$\theta$ is a skew brace homomorphism. 
We claim that $\theta$ maps generators of $E'$ to generators of $F'$. Indeed,
by using Remark \ref{TrivOfAnExtension}, 
\begin{multline*}
    \theta((\alpha(q,x)\alpha(q,-q)^{-1}\alpha(x,-x)^{-1}\alpha(q+x,-q)\alpha(q+x-q,-x),[q,x]))\\
    =(\eta(q,x)\eta(q,-q)^{-1}\eta(x,-x)^{-1}\eta(q+x,-q)\eta(q+x-q,-x),[q,x]).
\end{multline*}
Similarly,
\begin{multline*}
    \theta((\alpha(q,-q)^{-1}\alpha(x,-x)^{-1}\mu(q,x)\alpha(-q,q\circ x)\alpha(-q+q\circ x,-x),q*x))\\
    =(\eta\left(q,-q\right)^{-1}\eta(x,-x)^{-1}\gamma(q,x)\eta(-q,q\circ x)\eta(-q+q\circ x,-x),q*x).
\end{multline*}
Our candidate for the isomorphism is the map $\tilde{\theta} \colon E'\rightarrow F'$,  
the restriction of the map $\theta$ to the brace $E'$. Similarly, we construct 
a homomorphism $\omega\colon F\rightarrow \tilde{E}$, where $\tilde{E}=(\C^{\times})^m\times_{(\alpha,\mu)}Q$. The
restriction of $\omega$ to $F'$ is the inverse of $\tilde{\theta}$ because the compositions 
restrict to identities on the generators of $E'$ and $F'$. 

It follows that the diagram \eqref{eq:isoclinism} is commutative. 
\end{proof}
\begin{rem}
    As trivial skew braces can only be isoclinic to trivial skew braces, 
    the skew brace Schur covers of a perfect 
    group are exactly the ones coming from group theory.
\end{rem}
\begin{rem}
    Let $p$ be a prime number. 
    Example \ref{ex:Schurcovercyclicprime} implies that there 
    are two isoclinism classes of skew braces of order $p^2$. 
    One contains trivial skew braces and the other 
    contains $C_{(p^2,p)}$ and $\B_p$.
\end{rem}

\section{Representation Theory}

The following definition is inspired by 
\cite[Remark 3.3]{MR4504147}.

\begin{defn}
Let $A$ be a skew brace. A (complex degree-$n$) \emph{representation} of $A$ 
is a pair $(\beta,\rho)$ of group homomorphisms 
\begin{equation}
\label{eq:representation}
\beta\colon A_+\to\GL_n(\C)\quad\text{and}\quad 
\rho\colon A_\circ\to\GL_n(\C)
\end{equation}
such that 
$\beta(x\circ y)= \rho(x)\beta(y)\rho(x)^{-1}\beta(x)$ for all $x,y\in A$. 
\end{defn}

We also define \emph{projective representations} of 
skew braces considering group homomorphisms $A_+\to\PGL_n(\C)$ 
and $A_\circ\to\PGL_n(\C)$. 

\begin{notation}
   We will write $\prescript{\rho(x)}{}{\beta(y)}=\rho(x)\beta(y)\rho(x)^{-1}$ 
   to denote the conjugation 
   of $\beta(y)$ by $\rho(x)$. 
\end{notation}

\begin{rem}
\label{rem:representation}
    Equation \eqref{eq:representation} is equivalent to 
    \[
    \beta(x+\lambda_x(y)-x)= \rho(x)\beta(y)\rho(x)^{-1}=\prescript{\rho(x)}{}{\beta(y)}.
    \]
\end{rem}

\begin{defn}
    Let $A$ be a skew brace. 
    We say that two representations 
    $(\beta,\rho)$ and $(\beta_1,\rho_1)$ of $A$ 
    are \emph{equivalent} 
    if there exists $f\in\GL_n(\C)$ such that
    \[
    \beta_1(x)= f\beta(x) f^{-1}\quad 
    \text{and}
    \quad \rho_1(x)= f\rho(x) f^{-1} 
    \quad \text{for all $x\in A$}.
    \]
\end{defn}

We also define \emph{equivalence} of degree-$n$ projective representations 
of skew braces considering group homomorphisms $A_+\to\PGL_n(\C)$ 
and $A_\circ\to\PGL_n(\C)$ and a matrix $f\in\PGL_n(\C)$.

%
%


\begin{rem}
\label{rem:matrix_valued}
Let $(\beta,\rho)$ be a projective representation. 
By choosing coset representatives of $\beta(x)$ and $\rho(x)$ we can always 
think of a projective representation of a skew brace $A$ 
as a pair of matrix-valued function $T,U\colon A\to\GL_n(\C)$ 
such that 
\begin{gather*}
    T(x)T(y)=\alpha^{(\beta,\rho)}(x,y)T(x+y),
    \quad
    U(x)U(y)=\eta^{(\beta,\rho)}(x,y)U(x\circ y),\\
    U(x)T(y)U(x)^{-1}T(x)=\mu^{(\beta,\rho)}(x,y)T(x\circ y)
\end{gather*}
for all $x,y\in A$, where $\alpha^{(\beta,\rho)},\eta^{(\beta,\rho)},\mu^{(\beta,\rho)}\colon A\times A\to\C^{\times}$
are maps. 
By convention, the identity matrix 
will be the representative of $\beta(0)$ and $\rho(0)$.  
\end{rem}

In the context of Remark \ref{rem:matrix_valued}, two 
degree-$n$ projective representations 
with matrix-valued functions $(T,U)$ and $(T_1,U_1)$ of a skew brace $A$ 
are \emph{equivalent} 
if there exist $f\in\GL_n(\C)$ 
and maps 
$\theta,\gamma\colon A\to \C^{\times}$ such that
$T_1(x)=\theta(x)fT (x)f^{-1}$ and 
$U_1(x)=\gamma(x)fU (x)f^{-1}$ for all $x\in A$. This does not
depend on the choice of matrix-valued functions.

\begin{pro}
    \label{pro:projrepFS}
    Let $A$ be a finite skew brace and 
    $(\beta,\rho)$ be a projective representation of $A$. Then 
    $(\alpha^{(\beta,\rho)},\mu^{(\beta,\rho)})\in Z^2_b(A,\C^\times)$
    and 
    $\eta^{(\beta,\rho)}\in Z^2(A_\circ,\C^{\times})$.
\end{pro} 

\begin{proof}
    Let $(T,U)$ be a matrix-valued functions of $(\beta,\rho)$. 
    We will omit the indices of $\alpha$, $\mu$ and $\eta$. 
    Since  
    \[
    T(a)=T(0+a)=\alpha(0,a)^{-1}T(0)T(a), 
    \]
    it follows that $\alpha(0,a)=1$. Similarly, $\alpha(a,0)=1$. In addition,
    \begin{align*}
    \alpha(x,y+z)\alpha(y,z)T(x)T(y)T(z)&=T(x+(y+z))\\
    &=T((x+y)+z)\\
    &= \alpha(x+y,z)\alpha(x,y)T(x)T(y)T(z).
    \end{align*}
    Thus, $\alpha\in Z^2(A_+,\C^{\times})$. Similarly, one has $\eta\in Z^2(A_\circ,\C^{\times})$.
    Moreover,
    \begin{align*}
        & \mu(x,y\circ z)^{-1}\mu(y,z)^{-1} U(x)U(y)T(z)U(y)^{-1}T(y)U(x)^{-1}T(x)\\
        &=T(x\circ (y\circ z))\\
        &=T((x\circ y)\circ z)\\
        &= \mu(x\circ y,z)^{-1}\mu(x,y)^{-1}U(x)U(y)T(z)U(y)^{-1}T(y)U(x)^{-1}T(x).
    \end{align*}
    Finally, we have
    \begin{align*}
       & \alpha(x\circ y,\lambda_x(z))\alpha(x,\lambda_x(z))\mu(x,y)^{-1}\mu(x,z)^{-1}T(y)^{U(x)}T(z)^{U(x)}T(x)\\
        &= T(x\circ y-x+x\circ z)\\
        &=T(x\circ(y+z))\\
        &= \mu(x+y,z)^{-1}\alpha(y,z)^{-1}(T(y)T(z))^{U(x)}T(x).
    \end{align*}
    This gives us \eqref{eq:equivFS5} and concludes the proof.
    \end{proof}
    

\begin{defn}
    Let $A$ be a skew brace and 
    $\alpha,\mu \colon A\times A\to \C^{\times}$ be maps. A pair
    $(\beta,\rho)$, where 
    $\beta,\rho\colon A\to\GL_n(\C)$ are maps, 
    is an 
    $(\alpha,\mu)$-representation of $A$ if 
    \begin{align*}
        &\beta(0)=\rho(0)=\id,\\
        &\beta(x)\beta(y)=\alpha(x,y)\beta(x+y),\\
        &\rho(x)\rho(y)=\mu(x,y)\rho(x\circ y),\\
        &\rho(x)\beta(y)\rho(x)^{-1}\beta(x)=\mu(x,y)\beta(x\circ y),
    \end{align*}
    for all $x,y\in A$. 
\end{defn}

\begin{rem}
    An $(\alpha,\mu)$-representation is a matrix-valued function of a unique
    projective representation. We will use these two 
    ways of thinking about $(\alpha,\mu)$-representations 
    depending on our needs. 
\end{rem}

Let $A$ be a finite skew brace. We say that an 
$(\alpha,\mu)$-representation of $A$ is \emph{equivalent} to
an $(\alpha_1,\mu_1)$-representation of $A$ if 
they are equivalent as projective representations. 

\begin{rem}
    By Proposition \ref{pro:projrepFS}, 
    if there exists  an $(\alpha,\mu)$-representation $(\beta,\rho)$, 
    then $(\alpha,\mu)\in Z_b^2(A,\C^{\times})$. 
\end{rem}

\begin{pro}
    Let $A$ be a finite skew brace and 
    $(\beta,\rho)$ and $(\beta_1,\rho_1)$ be respectively $(\alpha,\mu)$- and $(\alpha_1,\mu_1)$-representations of $A$. 
    If they are equivalent, then $(\alpha,\mu)$ and $(\alpha_1,\mu_1)$ are cohomologous. Moreover, $(\alpha,\mu)$ is a coboundary if and only if $(\beta,\rho)$ is 
    equivalent to a representation.
\end{pro}

\begin{proof}
    Assume that there are an invertible matrix $f$ 
    and maps $\theta,\gamma\colon A\to \C^{\times}$ such that
    $\theta,\gamma\colon A\to \C^{\times}$ such that
    $\beta_1(x)=\theta(x)f\beta (x)f^{-1}$ and 
    $\rho_1(x)=\gamma(x)f\rho(x) f^{-1}$ for all $x\in A$. 
       Then $\beta(x)\beta(y)=\alpha(x,y)\beta(x+y)$ implies that
    \[
    \beta_1(x)\beta_1(y)=\alpha(x,y)\theta(x)\theta(y)\theta(x+y)\beta(x+y).
    \]
    Thus $\alpha_1=\alpha \partial_+ \theta$. In addition,  $\rho(x)\beta(y)\rho(x)^{-1}\beta(x)=\mu(x,y)\beta(x\circ y)$ 
    implies that
    \[\rho_1(x)\beta_1(y)\rho_1(x)^{-1}\beta_1(x)=\mu(x,y)\theta(x)\theta(y)\theta(x\circ y)^{-1}\beta(x\circ y).
    \]
    Therefore $\mu=\mu_1\partial_{\circ}\theta$.
    
    Assume that there is a map $\theta \colon A\to \C^{\times}$ such that $(\alpha,\mu)=(\partial_+ \theta,\partial_{\circ}\theta)$. Let 
    $\beta_2(x) =\theta(x)^{-1}\beta(x)$ and 
    $\rho_2(x) = \theta(x)^{-1}\rho(x)$. 
    Straightforward computations show that this pair of maps 
    is a linear representation.
\end{proof}


The following proposition is straightforward. 

\begin{pro}
\label{alphmualgebra}
    Let $A$ be a finite skew brace and 
    $(\alpha,\mu)\in Z^2_b(A,\C^{\times})$. 
    Consider the vector space $\C[A]$ 
    with basis $\{e_a:a\in A\}$. Let 
    \[
    e_a\cdot e_b=\alpha(a,b)e_{a+b},
    \quad
    e_a\circ e_b=\mu(a,b)e_{a\circ b},
    \quad 
    a,b\in A, 
    \]
    and extend these operations bilinearly to $\C[A]$. 
    Then 
    both $(\C[A],\cdot)$ and 
    $(\C[A],\circ)$ are algebras
    with the same neutral element 
    for multiplication. Moreover, 
    the elements $e_a$ with $a\in A$ are invertible with respect to both operations and 
    \[ 
    e_a^{-1}=\alpha(a,-a)^{-1}e_{-a},
    \quad 
    e_{a}'=\mu(a,a')^{-1}e_{a'}.\qedhere 
    \]
\end{pro}


\begin{defn}
    The structure of Proposition \ref{alphmualgebra} will be 
    called the complex 
    \emph{twisted skew brace algebra} of $A$ with respect
    to $(\alpha,\mu)$. It will be  
    denoted by $\C^{(\alpha,\mu)}[A]$.
\end{defn}

\begin{defn}
The twisted skew brace algebras $\C^{(\alpha,\mu)}[A]$ and $\C^{(\alpha_1,\mu_1)}[A]$ are called \emph{equivalent} if there exist a map 
$t\colon A\to \C^{\times}$ such that 
$\psi(e_a)=t(a)e_a$ 
for all $a\in A$ 
and a  
linear isomorphism $\psi\colon 
\C^{(\alpha,\mu)}[A]\to \C^{(\alpha_1,\mu_1)}[A]$ such that
for $i\in\{1,2\}$ the diagrams 
\[\begin{tikzcd}
	{\C^{(\alpha,\mu)}[A]\otimes \C^{(\alpha,\mu)}[A]} & {\C^{(\alpha,\mu)}[A]} \\
	{\C^{(\alpha_1,\mu_1)}[A]\otimes \C^{(\alpha_1,\mu_1)}[A]} & {\C^{(\alpha_1,\mu_1)}[A]}
	\arrow["{m_i}", from=1-1, to=1-2]
	\arrow["{\psi\otimes\psi }"', from=1-1, to=2-1]
	\arrow["\psi", from=1-2, to=2-2]
	\arrow["{m_i}", from=2-1, to=2-2]
\end{tikzcd}\]
corresponding to $m_1(a\otimes b)=a\cdot b$ and 
$m_2(a\otimes b)=a\circ b$, 
are commutative.
\end{defn}

\begin{pro}
    Two twisted skew brace algebras $\C^{(\alpha,\mu)}[A]$ and $\C^{(\alpha_1,\mu_1)}[A]$ are equivalent if and only if $(\alpha,\mu)$ and $(\alpha_1,\mu_1)$ are cohomologous.
\end{pro}

\begin{proof}
    Assume that there is an equivalence of twisted skew brace algebras with maps $\psi\colon \C^{(\alpha,\mu)}[A]\to \C^{(\alpha_1,\mu_1)}[A]$ and $t\colon A\to \C^{\times}$. Then 
    \[
    \psi(e_a\cdot e_b)=\alpha(a,b)t(a+b)e_{a+b},
    \quad 
    \psi(e_a)\cdot\psi(e_b)=t(a)t(b)\alpha_1(a,b)e_{a+b}
    \]
    for all $a,b\in A$. 
    Thus $\alpha=\partial_+t\alpha_1$. Similarly one shows that $\mu=\partial_{\circ}t\mu_1$.
    
    Conversely, assume that $(\alpha,\mu)=(\partial_+ t\alpha_1,\partial_{\circ}t\mu_1)$ for some $t\colon A\to\C^{\times}$. Routine calculations show that 
    the map 
    \[
    \psi\colon \C^{(\alpha,\mu)}[A]\to \C^{(\alpha_1,\mu_1)}[A],
    \quad
    e_a\mapsto t(a)e_a,
    \]
    is an equivalence of twisted skew brace algebras.
\end{proof}

If $M$ is a vector space, $\End(M)$ 
denotes the set of linear maps $M\to M$. 

\begin{defn}
    A \emph{module} over the twisted skew brace algebra $\C^{(\alpha,\mu)}[A]$ is a $\C$-vector space
    $M$ together with two algebra homomorphisms 
    \[
    \beta\colon (\C^{(\alpha,\mu)}[A],+,\cdot)\to\End(M),
    \quad
    \rho\colon (\C^{(\alpha,\mu)}[A],+,\circ)\to\End(M),
    \]
    such that 
    \begin{equation}
    \label{eq:modulecompcond}
    \beta(e_a\circ e_b)=\rho(e_a)\beta(e_b)\rho(e_a)^{-1}\beta(e_a)
    \end{equation}
    for all $a,b\in A$. 
\end{defn}

\begin{lem}
    \label{lem:repimpmod}
    Let $A$ be a finite skew brace and 
    $(M,\beta,\rho)$ be an $(\alpha,\mu)$-representation. 
    Define the maps 
    \begin{align*}
    &\overline{\beta}\colon\C^{(\alpha,\mu)}[A]\to\End(M), &&
    \overline{\beta}\left(\sum_{a\in A} \lambda_a e_a\right)=\sum_{a\in A}\lambda_a \beta(a),\\
    &\overline{\rho}\colon\C^{(\alpha,\mu)}[A]\to\End(M),
    &&
    \overline{\rho}\left(\sum_{a\in A} \lambda_a e_a\right)=\sum_{a\in A}\lambda_a \rho(a),
    \end{align*}
    Then $(M,\overline{\beta},\overline{\rho})$ 
    is a $\C^{(\alpha,\mu)}[A]$-module.
\end{lem}

\begin{proof}
    Since the maps $\overline{\rho}$ and $\overline{\beta}$ are obtained as linear extensions, they are vector spaces homomorphisms. Routine calculations show that $\overline{\beta}(e_a\cdot e_b)=\overline{\beta}(e_a)\overline{\beta}(e_b)$ and  $\overline{\rho}(e_a\circ e_b)=\overline{\rho}(e_a)\overline{\rho}(e_b)$ for all $a,b\in A$. Finally, 
    \begin{align*}
        \overline{\beta}(e_a\circ e_b)&=\mu(a,b)\beta(a\circ b)
        = \overline{\rho}(e_a)\overline{\beta}(e_b)\overline{\rho}(e_a)^{-1}\overline{\beta}(e_a)
    \end{align*}
    for all $a,b\in A$. 
\end{proof}

\begin{pro}
    \label{pro:bracealgrelation}
    Let $\C^{(\alpha,\mu)}[A]$ be a twisted skew brace algebra. 
    Then 
    \[
    e_a\circ(e_b\cdot v)= (e_a\circ e_b)\cdot e_a^{-1}\cdot (e_a\circ v)
    \]
    for all $a,b\in A$ and $v\in\C^{(\alpha,\mu)}[A]$. 
\end{pro}

\begin{proof}
By the linearity of the operations, it is enough to show the identity for $v=e_c$ for some $c\in A$. On the one hand,
\[e_a\circ(e_b\cdot e_c)= \mu(a,b+c)\alpha(b,c) e_{a\circ(b+c)}.
\]
On the other hand,
\[
(e_a\circ e_b)\cdot e_a^{-1}\cdot (e_a\circ e_c)= \mu(a,b)\alpha(a\circ b,\lambda_a(c))\alpha(x,\lambda_a(c))^{-1}\mu(a,c) e_{a\circ(b+c)}.
\]
This yields the equality by \eqref{eq:equivFS5}.
\end{proof}

\begin{exa}
    Let $n\geq1$. The direct product  $(\C^{(\alpha,\mu)}[A])^n$ is a $\C^{(\alpha,\mu)}[A]$-module.
\end{exa}


\begin{pro}
    Let $\C^{(\alpha,\mu)}[A]$ be a twisted skew brace algebra.
    There is a bijective correspondence between 
    $(\alpha,\mu)$-representations of $A$ and 
    $\C^{(\alpha,\mu)}[A]$-modules.
\end{pro}

\begin{proof}
    The direct implication is a consequence of Lemma \ref{lem:repimpmod}. Conversely, let $(M,\beta,\rho)$ 
    be a $\C^{(\alpha,\mu)}[A]$-module. By restricting the maps to the basis of $\C^{(\alpha,\mu)}[A]$ we obtain a projective representation of $A$.
\end{proof}

An important consequence of the previous results is that any factor set $(\alpha,\mu)$ is obtained from a projective representation. In fact this yields an important relation between $(\alpha,\mu)$-representations and coverings of 
finite skew braces.

\begin{defn}
Let 
\begin{equation}
    \label{eq:projective}
\begin{tikzcd}
0\arrow[r] & K \arrow[r] & A\arrow[r] & Q\arrow[r] & 0
\end{tikzcd}
\end{equation}
be an annihilator extension of finite skew braces. 
A complex $(\alpha,\mu)$-representation 
$(\beta,\rho)$ of $Q$  
can be \emph{lifted} to $A$ if there exists 
a representation $\tilde{\beta},\tilde{\rho}\colon A\to\GL_n(\C)$ such that the diagrams
\[
\begin{tikzcd}
A\arrow[r]\arrow[d,"\tilde{\beta}"] & Q \arrow[d]\arrow[d,"\beta"]\\
\GL_n(\C) \arrow[r] & \PGL_n(\C)
\end{tikzcd}
\begin{tikzcd}
A\arrow[r]\arrow[d,"\tilde{\rho}"] & Q \arrow[d]\arrow[d,"\rho"]\\
\GL_n(\C) \arrow[r] & \PGL_n(\C)
\end{tikzcd}
\]
commute; this representation is called a \emph{lift} 
of $(\beta,\rho)$. 
In addition, we say that the annihilator extension \eqref{eq:projective} 
has the \emph{lifting property} if 
every projective representation of $Q$ 
can be lifted to $A$. 
\end{defn}

\begin{thm} 
\label{thm:lift}
Let
\begin{equation}
\begin{tikzcd}
\label{eq:extensiontra}
0\arrow[r] & K \arrow[r] & A\arrow[r] & Q\arrow[r] & 0
\end{tikzcd}
\end{equation}
be an annihilator extension of finite skew braces. 
 The transgression map 
 \[
 \Tra \colon \Hom(K,\C^{\times})\to M_b(Q)
 \]
 is surjective if and only if the extension \eqref{eq:extensiontra} has the lifting property.
\end{thm}

\begin{proof}
Assume that $A=K\times_{(\alpha,\mu)}Q$. Let  $\beta_1,\rho_1\colon Q\to\GL_n(\C)$ be an $(\alpha_1,\mu_1)$-representation of $Q$.
 Assume that the transgression map is surjective.
 By assumption, there exist a group homomorphism $\phi\colon K\to \C^{\times}$ and a map $h\colon Q\to \C^{\times}$ such that $(\phi\alpha,\phi\mu)=(\alpha_1\partial_+ h,\mu_1\partial_{\circ}h)$. Straightforward computations show that the maps
 $\beta,\rho\colon A\to\GL_n(\C)$, 
 $\beta(q,x)=\phi(q)h(x)\beta_1(x)$, 
 $\rho(q,x)=\phi(q)h(x)\rho_1(x)$, 
 form a representation of $A$ which is a lift of $(\beta_1,\rho_1)$. 
 
Conversely, assume that the extension \eqref{eq:extensiontra} has the lifting property. Then there exist a representation $\beta,\rho\colon A\to\GL_n(\C)$ and maps $\gamma,\theta \colon A\to \C^{\times}$ such that $\beta(k,x)=\gamma(k,x)\beta_1(x)$ and $\rho(k,x)=\theta(k,x)\rho_1(x)$. Then 
\begin{align*}
\beta((k,x)\circ (l,y))&=\gamma(kl\mu(x,y),x\circ y) \beta_1(x\circ y),\\
\prescript{\rho(k,x)}{}{\beta(l,y)}\beta(k,x)&=\gamma(l,y)\gamma(k,x)\mu_1(x,y)\beta_1(x\circ y).
\end{align*}
Therefore
\begin{equation}
    \label{eq:liftimpsurj}
    \mu_1(x,y)=\gamma(l,y)^{-1}\gamma(kl\mu(x,y),x\circ y)\gamma(k,x)^{-1}.
\end{equation}
In addition, by computing both $\beta((k,x)+(l,y))$ and $\beta(k,x)\beta(l,y)$, one obtains that
\[
\alpha_1(x,y)=\gamma(l,y)^{-1}\gamma(kl\alpha(x,y),x+ y)\gamma(k,x)^{-1}.
\]
Thus $\overline{(\alpha_1,\mu_1)}\in \Ker(\Inf)$. By Theorem \ref{thm:HoschildSerre}, 
the transgression map is surjective.
\end{proof}

\subsection*{Acknowledgements}

This work was partially supported by 
the project OZR3762 of Vrije Universiteit Brussel and 
FWO Senior Research Project G004124N. Letourmy
is supported by FNRS. 
We thank the reviewer  
for suggesting Remark \ref{rem:representation}. 

\bibliographystyle{abbrv}
\bibliography{refs}

\end{document}